\documentclass[11pt]{article}
\usepackage[utf8x]{inputenc}
\usepackage{microtype}
\usepackage[T1]{fontenc}
\usepackage{ae,aecompl}
\usepackage{times}
\usepackage{float}
\usepackage{verbatim,amsmath,amsthm,amsfonts,amssymb,latexsym,graphicx,mathtools,extpfeil,color,mathabx}
\usepackage{stmaryrd}
\usepackage{epstopdf,pinlabel} 
\usepackage[all]{xy}
\usepackage{graphicx}
\usepackage{caption}
\usepackage{subcaption}
\DeclareMathAlphabet{\mathpzc}{OT1}{pzc}{m}{it}
\usepackage[colorlinks,pagebackref,hypertexnames=false]{hyperref} 
\usepackage{hyperref}
\hypersetup{
  colorlinks = true,
  linkcolor  = black
} 
\usepackage[alphabetic,backrefs,msc-links]{amsrefs}
\usepackage{accents}

\usepackage{multicol, color}

\definecolor{gold}{rgb}{0.85,.66,0}
\definecolor{cherry}{rgb}{0.9,.1,.2}
\definecolor{burgundy}{rgb}{0.8,.2,.2}
\definecolor{orangered}{rgb}{0.85,.3,0}
\definecolor{orange}{rgb}{0.85,.4,0}
\definecolor{olive}{rgb}{.45,.4,0}
\definecolor{lime}{rgb}{.6,.9,0}
\definecolor{green}{rgb}{.2,.7,0}
\definecolor{grey}{rgb}{.4,.4,.2}
\definecolor{brown}{rgb}{.4,.3,.1}

\newcommand\CS{{\rm CS}}
\newcommand\hrI{{\widehat {\rm I}}}
\newcommand\hrC{{\widehat { C}}}

\newcommand\crI{{\widecheck {\rm I}}}
\newcommand\crC{{\widecheck { C}}}

\newcommand\brI{{\overline {\rm I}}}
\newcommand\brC{{\overline {C}}}

\newcommand\mdeg{{\rm mdeg}}

\newcommand{\Z}{{\bf Z}}
\newcommand{\Q}{{\bf Q}}
\newcommand{\R}{{\bf R}}
\newcommand{\C}{{\bf C}}
\newcommand{\U}{{\rm U}}
\newcommand{\SU}{{\rm SU}}
\newcommand{\SO}{{\rm SO}}
\newcommand{\su}{\mathfrak{su}}
\newcommand{\fd}{{\mathfrak d}}
\newcommand{\fA}{{\mathfrak A}}
\newcommand{\fC}{{\mathfrak C}}
\newcommand{\fD}{{\mathfrak D}}
\newcommand{\fK}{{\mathfrak K}}
\newcommand{\fL}{{\mathfrak L}}
\newcommand{\fU}{{\mathfrak U}}
\newcommand{\loc}{{\rm loc}}
\newcommand{\dvol}{{\rm dvol}}
\usepackage{aliascnt}
\numberwithin{equation}{section}

\newcommand{\Addresses}{{
  \bigskip
  \footnotesize

  Aliakbar Daemi, \textsc{Simons Center for Geometry and Physics, State University of New York,
  Stony Brook, NY 11794}\par\nopagebreak
  \textit{E-mail address}: \texttt{adaemi@scgp.stonybrook.edu}
}}

\theoremstyle{plain}

\def\makeautorefname#1#2{\AtBeginDocument{\expandafter\def\csname#1autorefname\endcsname{#2}}}

\newcommand{\mynewtheorem}[2]{
  \newaliascnt{#1}{equation}          
  \newtheorem{#1}[#1]{#2}
  \aliascntresetthe{#1}
  \makeautorefname{#1}{#2}
}
\mynewtheorem{theorem}{Theorem}
\mynewtheorem{prop}{Proposition}
\mynewtheorem{cor}{Corollary}
\mynewtheorem{construction}{Construction}
\mynewtheorem{lemma}{Lemma}
\mynewtheorem{conjecture}{Conjecture}
\mynewtheorem{question}{Question}
\mynewtheorem{problem}{Problem}

\numberwithin{substep}{step}
\makeautorefname{step}{Step}
\makeautorefname{substep}{Step}

\numberwithin{subcase}{case}
\makeautorefname{case}{Case}
\makeautorefname{subcase}{case}

\theoremstyle{remark}
\mynewtheorem{remark}{Remark}

\theoremstyle{definition}
\mynewtheorem{definition}{Definition}
\mynewtheorem{example}{Example}
\mynewtheorem{exercise}{Exercise}
\mynewtheorem{convention}{Convention}
\newtheorem*{convention*}{Convention}
\newtheorem*{conventions*}{Conventions}

\newtheorem{theorem-intro}{Theorem}
\newtheorem{cor-intro}{Corollary}
\newtheorem{prop-intro}{Proposition}

\DeclareMathOperator{\tr}{tr}
\DeclareMathOperator{\ind}{index}

\def\({\mathopen{}\left(}
\def\){\right)\mathclose{}}

\title{\large \bf Chern-Simons Functional and the Homology Cobordism Group}
\author{\bf \sc \large Aliakbar Daemi\thanks{The work of the author was supported by NSF Grant DMS-1812033.}}
\begin{document}

\maketitle
\begin{abstract}
	For each integral homology sphere $Y$, a function $\Gamma_Y$ on the set of integers is constructed. It is established that $\Gamma_Y$ depends only on the homology cobordism 
	of $Y$ and it recovers the Fr{\o}yshov invariant. A relation between $\Gamma_Y$ and Fintushel-Stern's $R$-invariant is stated. It is shown that the value of $\Gamma_Y$ at each integer is related to the critical values of the Chern-Simons functional. Some topological 
	applications of $\Gamma_Y$ are given. In particular, it is shown that if $\Gamma_Y$ is trivial, then there is no simply connected homology cobordism from $Y$ to itself.
\end{abstract}
\maketitle

\section{Introduction}

Various Floer homology theories provide powerful tools in 3-manifold topology \cite{Fl:I,OzSz:HF,KM:monopoles-3-man}. The definitions of these invariants follow a similar pattern. To a given 3-manifold $Y$, one associates a pair of an infinite dimensional space $\mathcal B(Y)$ and a functional $\CS$, defined on $\mathcal B(Y)$.\footnote{Here $\CS$ stands for the Chern-Simons functional, which is the relevant functional in the case of instanton Floer homology. This is not a standard notation for other 3-manifold Floer homologies. We use this notation to emphasize on formal similarities among the definitions of these 3-manifold Floer homology theories.} Then the relevant Floer homology of $Y$ is obtained by applying the Morse homological methods to the functional $\CS$. A unique feature of {\it instanton homology} \cite{Fl:I} among other Floer homology theories is that both $\mathcal B(Y)$ and the functional $\CS$ are topological. On the other hand, one needs to fix additional auxiliary structures for {\it monopole Floer homology} (in the form of a Riemannian metric for $Y$) and {\it Heegaard Floer homology} (including a Heegaard diagram for $Y$) to define $\mathcal B(Y)$ and $\CS$. In the present article, we exploit this property of instanton Floer homology to introduce an invariant of integral homology 3-spheres which is preserved by {\it homology cobordisms}. The definition of this invariant is partly inspired by ideas from Min-Max theory.\footnote{See Appendix \ref{min-max-app} for an elaboration on this point.} This homology cobordism invariant provides a platform to unify works of various authors including Donaldson \cite{Don:neg-def-gauge}, Fintushel-Stern \cite{FS:pseudofree,FS:HFSF}, Fr{\o}yshov \cite{Fro:h-inv}, Furuta \cite{Fur:hom-cob}, Hedden-Kirk \cite{KH:Wh-dble}.

\subsection{Statement of Results}
Let $Y$ and $Y'$ be two integral homology spheres. A cobordism $W$ from $Y$ to $Y'$ is a smooth 4-manifold with boundary $-Y\sqcup Y'$. The 3-manifolds $Y$ and $Y'$ are {\it homology cobordant}, if there is a cobordism $W$ from $Y$ to $Y'$ such that $H_*(W,Y;\Z)=H_*(W,Y';\Z)=0$. The collection of all integral homology spheres modulo homology cobordism relation is called the {\it homology cobordism group} and is denoted by $\Theta_\Z^3$.

Suppose $Y$ is an integral homology sphere. As the main construction of the present article, we introduce a function $\Gamma_Y:\Z \to \overline {\R}^{\geq 0}$, where $ \overline {\R}^{\geq 0}$ is the extended positive real line $\R^{\geq 0}\cup \{\infty\}$, equipped with the obvious ordering. The following theorem states some of the basic properties of this function:
\begin{theorem-intro}\label{basic-prop}
	The function $\Gamma_Y$ satisfies the following properties:
	\begin{itemize}
		\item[(i)] $\Gamma_Y$ is non-decreasing.
		\item[(ii)] If there is a homology cobordism from the integral homology sphere $Y$ to another 
			integral homology sphere $Y'$, then $\Gamma_Y=\Gamma_{Y'}$.
	\end{itemize}
\end{theorem-intro}

Both parts of the above theorem can be strengthened. Before stating an improvement of Theorem \ref{basic-prop} (i), we need to give a definition:

\begin{definition}
	An integral homology sphere $Y$ is {\it $\SU(2)$-non-degenerate}, if all $\SU(2)$ flat connections 
	on $Y$ are non-degenerate. That is to say, for any flat connection $\alpha$ on $Y$, we have $H^1(Y;ad_\alpha)=0$
	where $ad_\alpha$ is the flat vector bundle of rank $3$ associated to the adjoint representation of $\alpha$.
\end{definition}

\begin{theorem-intro}\label{basic-prop-2}
	Suppose $Y$ is an integral homology sphere. There is a positive constant $\tau(Y)$ such that for any 
	positive integer $i$, we have: \[\Gamma_Y(i)\geq \tau(Y).\]
	Moreover, if $Y$ is $\SU(2)$-non-degenerate, then there is a positive constant $\tau'(Y)$ such that for any positive 
	integer $i$, we have:
	\[\Gamma_Y(i+1)\geq \Gamma_Y(i)+\tau'(Y).\]
\end{theorem-intro}

The constants $\tau(Y)$ and $\tau'(Y)$ in the above theorem can be explicitly defined in terms of the moduli spaces of anti-self-dual $\SU(2)$-connections on $\R\times Y$. (See Definitions \ref{tau} and \ref{tau'}.) The following theorem gives a generalization of part (ii) of Theorem \ref{basic-prop}:

\begin{theorem-intro} \label{neg-def}
	Suppose $W$ is a cobordism from an integral homology sphere $Y$ to another homology sphere $Y'$ 
	such that {$b_1(W)=0$} and the intersection form of $W$ is negative definite, i.e., $b^+(W)=0$. Then there is a non-negative constant $\eta(W)$ such that for any positive integer $i$ and non-positive integer $j$, we have\footnote{Here we define $\infty-k$ to be $\infty$
	for $k\in \overline {\R}^{\geq 0}$.}: 
	\[\Gamma_{Y'}(i)\leq \Gamma_{Y}(i)-\eta(W)\hspace{1cm}\Gamma_{Y'}(j) \leq \max(\Gamma_{Y}(j)-\eta(W),0).\]
	Moreover, the constant $\eta(W)$ is positive unless there is an $\SU(2)$-representation of $\pi_1(W)$ which extends non-trivial representations of $\pi_1(Y)$ and $\pi_1(Y')$.
\end{theorem-intro}

Gauge theoretical methods provide an important source of tools to study the group $\Theta^3_\Z$. In his groundbreaking work \cite{Fro:h-inv}, Fr{\o}yshov introduced a homomorphism $h:\Theta^3_\Z \to \Z$ which is defined using Floer's instanton homology of integral homology spheres. Fr{\o}yshov's construction motivated the definition of numerical invariants of homology cobordism group in the other Floer homologies \cite{Fr:SW-4-man-bdry,OzSz:d-inv,KM:monopoles-3-man,Fro:mon-h-inv,Man:tri-pin(2),Lin:MB,HM:inv-HF}, which have many interesting applications in low dimensional topology. Invariants with similar flavor are introduced in \cite{St:conn,HHL:conn} by building on the constructions of \cite{Man:tri-pin(2),HM:inv-HF}. The invariant $\Gamma_Y$ can be regarded  as a refinement of Fr{\o}yshov's $h$-invariant:
\begin{theorem-intro}\label{support-h-inv}
	The function $\Gamma_Y$ takes a finite value at an integer $k$ if and only if $k\leq 2h(Y)$.
\end{theorem-intro}

The invariant $h(Y)$ gives consraints on the intersection form of 4-manifolds $X$ which fill the 3-manifold $Y$. For example, if $h(Y)\leq 0$, then there is no negative definite 4-manifold $X$ with boundary $Y$ such that the intersection form of $W$ is not diagonal \cite[Theorem 3]{Fro:h-inv}. If we let $Y=S^3$, then this result specializes to Donaldson's groundbreaking {\it diagonalizability theorem} \cite{Don:neg-def-gauge}. Negative definite 4-manifolds can be also used to obtain constraints on $\Gamma_Y$:

\begin{theorem-intro}\label{NSIF-neg-def}
	Suppose $X$ is a 4-manifold whose boundary is an integral homology sphere $Y$. Suppose 
	the intersection form $Q$ of $X$ on $H^2(X;\Z)/{\rm Tor}$ has the following form
	\[
	  Q=(-1)\oplus \dots\oplus (-1)\oplus \mathcal L.
	\]
	Here $\mathcal L$ is a non-trivial negative definite lattice such that:
	\begin{equation} \label{m(l)}
	  m(\mathcal L):=\min_{\alpha\in \mathcal L,\,\alpha\neq 0}\{|Q(\alpha)|\}
	\end{equation}
	is greater than $1$. Then we have:
	\begin{equation}
	  	\hspace{3cm}\Gamma_Y(i)\leq \frac{m(\mathcal L)}{4} \hspace{1cm}i\leq \lfloor \frac{m(\mathcal L)}{2}\rfloor
	\end{equation}
	In particular, $\Gamma_{Y}(1)$ is a finite number.
\end{theorem-intro}
For a slightly more general version of Theorem \ref{NSIF-neg-def} see Propositions \ref{bounding-reducible} and \ref{bounding-reducible-2}. 

In the case that $Y=S^3$, we have:
\begin{equation} \label{triv-Gamma}
  \Gamma_{S^3}(k)=\left\{
  \begin{array}{ll}
 	\infty&k> 0\\
	0&k\leq 0\\
  \end{array}
  \right.	
\end{equation}
This is an immediate consequence of the definition of $\Gamma_Y$. It also follows from Theorem \ref{neg-def} applied to the product cobordism between two copies of $S^3$ and Theorem \ref{support-h-inv}. For an integral homology sphere, we say $\Gamma_Y$ is trivial, if $\Gamma_Y=\Gamma_{S^3}$. Theorems \ref{neg-def} and \ref{NSIF-neg-def} imply that:
\begin{cor-intro}\label{simp-hom-cob}
	Let $W$ be a homology cobordism from $Y$ to $Y'$ such that $\Gamma(Y)$ is non-trivial 
	(and hence $\Gamma(Y')$ is non-trivial). Then the inclusion of $Y$ and $Y'$ in $W$ induce non-trivial maps of fundamental groups. In particular, there is not a 
	simply-connected homology cobordism from $Y$ to itself if $h(Y)\neq 0$.
\end{cor-intro}

This corollary is an extension of \cite[Proposition 1.7]{Tau:gauge-per}, which was originally proved using gauge theory on manifolds with periodic ends. (See Remark \ref{man-per-ends} about the relation between Corollary \ref{simp-hom-cob} and \cite{Mas:per}.) 
It also answers a variant of the following question asked by Akbulut \cite[Problem 4.95]{Ki:problem-list}. In this question, $\mu$ denotes the Rokhlin homomorphism:
\begin{question}\label{Akbulut-question}
	Let $W$ be a homology cobordism from $Y$ to $Y'$ such that $\mu(Y)\neq 0$ (and hence $\mu(Y')\neq 0$). Then $W$ is not simply connected.
\end{question}

\begin{remark}[Levin and Lidman]
	The following question was raised in \cite[Remark 1.13]{HLL:con-hom}: given an integral homology sphere $Z$ 
	which bounds a homology 4-ball and $\gamma\in \pi_1(Z)$, does there exist a homology 4-ball $X$ such that $\partial X=Z$ and $\gamma$ is 
	null-homotopic in $X$?
	Corollary \ref{simp-hom-cob} may be used to give an affirmative answer to this question.
	Suppose $Y$ is an integral homology sphere with a weight 
	one fundamental group. Suppose $\gamma$ is a closed curve in $Y$ which normally generates $\pi_1(Y)$. 
	We also assume that $\Gamma_Y\neq \Gamma_{S^3}$. For example, we may take 
	$Y$ to be the poincar\'e homology sphere $\Sigma(2,3,5)$. (See Example \ref{simple-examples}.) The 3-manifold $Z=Y\#-Y$ bounds a homology 4-ball.
	However, there is no homology 4-ball $X$ with $\partial X=Z$ such that $\gamma$ vanishes in $\pi_1(X)$ because Corollary \ref{Akbulut-question} 
	asserts that the inclusion map induces a non-trivial map from $\pi_1(Y)$ (and also $\pi_1(-Y)$) to $\pi_1(X)$.
\end{remark}

Fintushel and Stern introduced an invariant for any Seifert fibered homology sphere $\Sigma(a_1,\dots,a_n)$ in \cite{FS:pseudofree}, which is denoted by $R(a_1,\dots,a_n)$ and is defined as follows:
\begin{equation}\label{Ra1an}
	R(a_1,\dots,a_n)=\frac{2}{a}-3+n+\sum_{i=1}^n\frac{2}{a_i}\sum_{k=1}^{a_i-1}\cot(\frac{\pi ka}{a_i^2})
	\cot(\frac{\pi k}{a_i})\sin^2(\frac{\pi k}{a_i}).
\end{equation}
Here $a=a_1\cdot a_2 \dots a_n$. It turns out that the above number is an odd integer, not smaller than $-1$. The following theorem states that $\Gamma_Y$ is related to Fintushel and Stern's $R$-invariant:
\begin{theorem-intro}\label{seifert-space-comp-thm}
	Let $Y$ be an integral homology sphere which has the form:
	\[Y=Y_1\#Y_2 \dots\#Y_k\]
	where $Y_i$ are (not necessarily distinct) Seifert fibered homology spheres with positive values of the 
	$R$-invariant. 
	Let $Y_1$ be the Seifert fibered space $\Sigma(a_1,a_2,\dots a_n)$ such that the value of the product 
	$a_1a_2 \dots a_n$ is maximum 
	among all Seifert spaces $Y_i$.  
	Then for $1\leq i \leq \lfloor \frac{R(a_1,\dots,a_n)+3}{4}\rfloor$, we have:
	\[
	  \Gamma_{Y}(i)=\frac{1}{4a_1a_2\dots a_n}.
	\]
	In particular, $h(\Sigma (a_1,\dots,a_n))\geq \frac{1}{2}\lfloor \frac{R(a_1,\dots,a_n)+3}{4}\rfloor$.
\end{theorem-intro}
The main theorem of \cite{FS:pseudofree} asserts that if $R(a_1,\dots,a_n)>0$, then there is no negative definite 4-manifold $X$ whose boundary is $-\Sigma(a_1,\dots,a_n)$. This result follows immediately from Theorems \ref{neg-def} and \ref{seifert-space-comp-thm}. The following well-known theorem of Furuta is another corollary of Theorem \ref{seifert-space-comp-thm}. 

\begin{cor-intro}[\cite{Fur:hom-cob}]
	Suppose $\{Y_i=\Sigma(a_{i,1}, \dots, a_{i,n_i})\}_{i\in I}$ is a collection of Seifert fibered homology spheres with 
	positive $R$-invariants such that the positive integers $a_i:=a_{i,1} a_{i,2} \dots a_{i,n_i}$ are distinct.
	Then the integral homology spheres $Y_i$ determine linearly independent elements of $\Theta_\Z^3$.	
\end{cor-intro}

As a special case, Furuta's result implies that the integral homology spheres $\{\Sigma(p,q,pqk-1)\}_{k\in \Z^{>0}}$, for coprime positive integers $p$, $q$, span a subgroup of $\Theta_\Z^3$ isomorphic to $\Z^{\infty}$ because $R(p,q,pqk-1)=1$.

\begin{cor-intro} \label{whitehead-dble}
	Suppose $D(T_{p,q})$ denotes the Whitehead double of the $(p,q)$-torus knot and 
	$Y_{p,q}$ is the 3-manifold $-\Sigma(D(T_{p,q}))$, the branched double cover of 
	$D(T_{p,q})$ with the reverse orientation.
	Then we have:
	\begin{equation}\label{Ypq-ineq}
	  \frac{1}{4pq(4pq-1)} \leq \Gamma_{Y_{p,q}}(1)\leq \frac{1}{4pq(2pq-1)}
	\end{equation}
	More generally, if $Y=n_1\cdot Y_{p_1,q_1}\#\dots \#n_k\cdot Y_{p_k,q_k}$ for integers $n_i$, $p_i$, $q_i$ 
	such that $p_i, q_i$ are coprime integers greater than $1$ and  $n_i$ is a positive integer, then:
	\begin{equation}\label{Ypq-ineq-gen}
		\min_{1\leq i\leq k}{\frac{1}{4p_iq_i(4p_iq_i-1)}} \leq 
		\Gamma_{Y}(1)\leq \min_{1\leq i\leq k}{\frac{1}{4p_iq_i(2p_iq_i-1)}}
	\end{equation}
\end{cor-intro}
This corollary is a consequence of Theorems \ref{basic-prop-2} and \ref{seifert-space-comp-thm} using the ingredients provided by \cite{KH:Wh-dble}. In \cite{KH:Wh-dble}, a lower bound for the positive values of the Chern-Simons functional on the set of flat $\SU(2)$-connections of $Y_{p,q}$ is given. This lower bound allows us to obtain the lower bound for $\Gamma_Y$ in \eqref{Ypq-ineq-gen}. The upper bound in \eqref{Ypq-ineq-gen} is verified with the aid of a negative definite cobordism from $\Sigma(p,q,2pq-1)$ to $Y_{p,q}$. In fact, Proposition \ref{CS-val-intro} below implies that $\Gamma_{Y_{p,q}}(1)$ is one of the following values and it is natural to ask which of these values are equal to $\Gamma_{Y_{p,q}}(1)$:
\[
  \frac{1}{4pq(4pq-1)}\hspace{1cm}\frac{1}{2pq(4pq-1)}\hspace{1cm}\frac{1}{4pq(2pq-1)}.
\]
Corollary \ref{whitehead-dble} can be used to conclude the following theorem proved in \cite{KH:Wh-dble}.
\begin{cor-intro}[\cite{KH:Wh-dble}] \label{whitehead-dble-2}
	The knots $\{D(T_{2,2^n-1})\}_{n\geq 2}$ are linearly independent in the smooth concordance group.
\end{cor-intro}

A stronger version of Corollary \ref{whitehead-dble-2} for a more general family of knots is proved in \cite{P:ind-br}. Analogous to Corollary \ref{whitehead-dble-2}, it is possible to reformulate the results of \cite{P:ind-br} in terms of $\Gamma_Y$. In particular, the arguments of \cite{P:ind-br} can be used to obtain information about $\Gamma_Y$ where $Y$ is the branched double cover of $S^3$, branched along one of the knots studied in \cite{P:ind-br}.

There are properties and invariants of integral homology spheres which do not respect the homology cobordism relation. But we can use $\Gamma_Y$ to show that the homology cobordism group is not completely blind to them. For instance, being an $\SU(2)$-non-degenerate integral homology sphere is not preserved by homology cobordisms. If $Y$ has an irreducible $\SU(2)$ flat connection, then $Y\#-Y$ is $\SU(2)$-degenerate and it is homology cobordant to $S^3$, an $\SU(2)$-non-degenerate integral homology sphere. Nevertheless, we have the following corollary of Theorems of \ref{basic-prop-2} and \ref{seifert-space-comp-thm}:
\begin{cor-intro}
	Suppose $\Sigma(a_1,a_2,\dots a_n)$ is a Seifert fibered space with $R(a_1,a_2,\dots a_n)\geq 5$.
	Then $\Sigma(a_1,a_2,\dots a_n)$ is not homology cobordant to an 
	$\SU(2)$-non-degenerate integral homology sphere.
	In particular, $\Sigma(a_1,a_2,\dots a_n)$ is not homology cobordant to a Brieskorn homology sphere $\Sigma(p,q,r)$.
\end{cor-intro}
\noindent
Note that there are Seifert fibered homology spheres $\Sigma(a_1,a_2,\dots a_n)$ with arbitrarily large values of $R(a_1,a_2,\dots a_n)$. (See Example \ref{large-R}.)

The Chern-Simons functional of an integral homology sphere $Y$ takes finitely many values on the space of flat $\SU(2)$-connections. The homology cobordism relation might change this set. For example, the connected sum $Y\#-Y$, which is homology cobordant to $S^3$, can take non-trivial values on the space of flat $\SU(2)$-connections. On the other hand, the following proposition implies that if $Y$ is homology cobordant to $Y'$, then the values of Chern-Simons functionals of $Y$ and $Y'$ on the space of flat connections share the set of finite values in the image of $\Gamma_Y$ and $\Gamma_{Y'}$:
\begin{prop-intro}\label{CS-val-intro}
	For any integral homology sphere $Y$ and any integer $k$, either $\Gamma_Y(k)=\infty$, 
	$\Gamma_Y(k)=0$, or there is an irreducible 
	flat connection $\alpha$ such that $\Gamma_Y(k)$ is equal to $\CS(\alpha)$ mod $\Z$. 
\end{prop-intro}

It is natural to ask whether $\Gamma_Y$ takes any irrational value. Proposition \ref{CS-val-intro} implies that if $Y$ is a linear combination of Seifert fibered homology spheres (or more generally plumbed 3-manifolds), then $\Gamma_Y$ is always rational valued. In order to find $Y$ with an irrational value in the image of $\Gamma_Y$, we firstly need to find an $\SU(2)$-flat connection $\alpha$ on an integral homology sphere such that $\CS(\alpha)$ is an irrational number. Even existence of such flat connections is an open question.

We end this part of the introduction by discussing a filtration on $\Theta_\Z^3$. Given two integral homology spheres $Y$, $Y'$, we define $Y\succeq_\Gamma Y'$ if $\Gamma_Y(1)\leq \Gamma_{Y'}(1)$. This defines a total quasi-order\footnote{A quasi-order $\succeq$ is a reflexive and transitive binary relation. A total quasi-order on a set $S$ is a quasi-order $\succeq$ on $S$ such that for any two elements $a$, $b$ of $S$ at least one of the relations $a\succeq b$ or $b\succeq a$ holds.} on $\Theta_\Z^3$. There is also a well-known quasi-order $\succeq_N$ on $\Theta_\Z^3$ where $Y\succeq_N Y'$ if there is a negative definite manifold cobordism $W$ from $Y'$ to $Y$. Theorem \ref{neg-def} implies that if $Y\succeq_N Y'$, then $Y\succeq_\Gamma Y'$. (If $W:Y\to Y'$ is a negative definite cobordism, then performing surgery on a set of loops representing a basis for $H_1(X,\R)$ gives rise to a negative definite cobordism from $Y$ to $Y'$ with vanishing $b_1$.) Theorem \ref{seifert-space-comp-thm} implies that if we pick the sequence $\{Y_k:=\Sigma(p,q,pqk-1)\}_{k}$, then for any integer $N$ we have $Y_{k}\succeq_\Gamma N\cdot Y_{k-1}$. In light of \cite{J:inf-sum}, it would be interesting to study the behavior of this filtration (or some refined version of it) with respect to connected sums of 3-manifolds. Of course, this requires a better understanding of the invariant $\Gamma_Y$ with respect to connected sums which will be investigated in \cite{D:conn-Gamma}.

\subsection{Outline of Contents}
For the purpose of this paper, we need to work with a version of instanton Floer homology which is defined with coefficients in a Novikov field. This version of instanton Floer homology is discussed in Section \ref{ins-Flo}. Subsection \ref{tilde-func} is devoted to a review of the functoriality of instanton Floer homology with respect to negative definite cobordisms with vanishing $b_1$. In \cite{Don:YM-Floer}, various {\it equivariant instanton Floer homology theories} for integral homology spheres are introduced. In Subsection \ref{equiv}, we introduce three such equivariant theories and show that they fit into an exact triangle. The construction of this exact triangle is inspired by similar objects in the context of monopole Floer homology \cite{KM:monopoles-3-man} and Heegaard Floer homology \cite{OzSz:HF}. 

The definition of $\Gamma_Y$ is given in Subsection \ref{GammaY} using the exact triangle of equivariant instanton Floer homologies. As it is explained there, one can avoid equivariant theories in the definition of $\Gamma_Y$. However, I find it more instructive to use the language of equivariant instanton Floer theories. I also believe that this approach would be more efficient in studying the behavior of $\Gamma_Y$ with respect to topological constructions such as surgery along a knot and taking connected sum. In Subsection \ref{prop-GammaY}, we verify the basic properties of $\Gamma_Y$ claimed above. In Section \ref{R-inv}, we study the relation between $\Gamma_Y$ and Fintushel and Stern's $R$ invariant. Section \ref{diag-gen} is devoted to the proof of Theorem \ref{NSIF-neg-def} and its generalizations. 

{\it Acknowledgements.} This work is partly inspired by the ongoing collaboration of the author with Kenji Fukaya. I am grateful to him for many enlightening discussions. The definition of the exact triangle of Subsection \ref{equiv} is motivated by the author's discussions with Michael Miller. I am thankful to him for explaining to me his work on instanton Floer homology of rational homology spheres. I thank Masaki Taniguchi for brining his work to the author's attention. I am also grateful to Simon Donaldson, Peter Kronheimer and Christopher Scaduto for many interesting conversations about the present work.

\section{A Review of Instanton Floer Homology}\label{ins-Flo}

\subsection{Instanton Floer Chain Complexes} \label{C-tilde}

Suppose $Y$ is an integral homology sphere and $P$ is an $\SU(2)$-bundle on $Y$, which is necessarily a trivial bundle. Let $\mathcal A(Y)$ be the space of $\SU(2)$-connections on $P$. We fix a trivialization of $P$ and denote the associated trivial connection by $\Theta$. Other connections on $P$ are given by adding elements of $\Omega^1(Y,\su(2))$ to the trivial connection. Here $\Omega^1(Y,\su(2))$ is the space of smooth 1-forms on $Y$ with coefficients in $\su(2)$, the Lie algebra of $\SU(2)$. For a connection $A=\Theta+a$, with $a$ being an element of $\Omega^1(Y,\su(2))$, the Chern-Simons functional of $A$ is defined to be:
\begin{equation*}
	\widetilde {\CS}(A):=-\frac{1}{8\pi^2}\int_Y \tr(a\wedge da+\frac{2}{3}a\wedge a\wedge a)
\end{equation*} 
Suppose $\mathcal G(Y)$ is the space of smooth automorphisms of $P$. Given an element of $g \in \mathcal G(Y)$, we can pull-back a connection using the automorphism $g^{-1}$. This determines an action of $\mathcal G(Y)$ on $\mathcal A(Y)$. An element of $\mathcal A(Y)$ is irreducible if its stabilizer with respect to the action of $\mathcal G(Y)$ is $\pm {\rm id}$. The action of each element of $\mathcal G(Y)$ changes the value of the Chern-Simons functional by an integer. Therefore, we have an induced map $\CS:\mathcal B(Y) \to \R/\Z$ where $\mathcal B(Y)$ is the quotient space $\mathcal A(Y)/\mathcal G(Y)$.

Instanton Floer homology of $Y$ can be regarded as the ``Morse homology group'' associated to the Chern-Simons functional $\CS$. The critical points of $\CS$ are represented by flat $\SU(2)$-connections on $Y$ and form a compact subspace of $\mathcal B(Y)$. The trivial connection (more precisely, the class represented by $\Theta$) is a singular element of $\mathcal B(Y)$, because the stabilizer of $\Theta$ consists of constant automorphisms of $P$. On the other hand, all non-trivial flat connections on $P$ are irreducible.
The trivial connection is a non-degenerate critical point of $\CS$, namely, Hessian of $\CS$ at $\Theta$ is non-degenerate modulo the action of the gauge group. For now, we assume that the other critical points are also non-degenerate and hence isolated. Consequently, there are only finitely many critical points of $\CS$.

We can form the analogues of downward gradient flow lines for the Chern-Simons functional. Fix two critical points $\alpha$ and $\beta$ of $\CS$ which are represented by flat connections $\widetilde \alpha, \widetilde \beta\in\mathcal A(Y)$. Suppose $A_0$ is a smooth  connection on $\R \times Y$, which is equal to the pull-backs of $\widetilde \alpha$ and $\widetilde \beta$ on the ends $(-\infty,-1]\times Y$ and $[1,\infty)\times Y$, respectively. The connection $A_0$ determines a path $z$ in $\mathcal B(Y)$ from $\alpha$ to $\beta$. Fix an integer $l\geq 3$ and define $\mathcal A_z(\alpha,\beta)$ to be the set of Sobolev connections $A$ on $\R \times Y$ such that $|\!|A-A_0|\!|_{L_l^2}< \infty$. As the notation suggests, this space depends only on the path $z$. We define the topological energy of an element $A \in \mathcal A_z(\alpha,\beta)$ to be:
\begin{equation*}
	\frac{1}{8\pi^2}\int_{\R \times Y } \tr(F(A)\wedge F(A))
\end{equation*}
The topological energy is independent of $A$, and we will denote it by $\mathcal E(z)$. It is also related to the Chern-Simons functional as follows:
\begin{equation} \label{energy-CS}
	\mathcal E(z)\equiv \CS(\alpha)-\CS(\beta) \mod \Z
\end{equation}
Suppose $\mathcal G_z(\alpha,\beta)$ denotes the group of automorphisms $g$ of the trivial $\SU(2)$-bundle on $\R \times Y$ such that $|\!|\nabla_{A_0} g|\!|_{L^2_l}<\infty$. This group acts on $\mathcal A_z(\alpha,\beta)$ and the quotient is denoted by $\mathcal B_z(\alpha,\beta)$.

Fix a metric on $Y$ and equip $\R \times Y$ with the product metric. A connection $A\in \mathcal A_z(\alpha,\beta)$ is {\it Anti-Self-Dual}, if the self-dual part of its curvature, $F^+(A)$, vanishes. The path $\{A|_{\{t\}\times Y}\}_{t\in \R}$, for an ASD connection $A$, can be regarded as the downward gradient flow line of the Chern-Simons functional where we use the following metric on $\mathcal A(Y)$: 
\begin{equation*}
	\hspace{3cm}\langle a,b \rangle:=\frac{1}{8\pi^2}\int_{Y } -\tr(a\wedge *b) \hspace{1cm} a,\,b \in \Omega^1(Y,\su(2))
\end{equation*}
The ASD equation is invariant with respect to the action of $\mathcal G_z(\alpha,\beta)$ and the quotient space of the space of solutions is denoted by $\mathcal M_z(\alpha,\beta)$. The space $\mathcal M_z(\alpha,\beta)$ is non-empty only if $\mathcal E(z)>0$ or $z$ is homotopic to the constant path. In the latter case, $\mathcal E(z)=0$ and $\alpha=\beta$. For now, we assume that solutions of the ASD equation are {\it regular}. That is to say, they are cut down transversely by the ASD equation. Then $\mathcal M_z(\alpha,\beta)$ is an orientable smooth manifold. 

To any $\alpha$ in $\mathcal B(Y)$, we can associate $\deg^+(\alpha), \deg^-(\alpha) \in \Z/8\Z$ such that \cite{Fl:I,Don:YM-Floer}:
\begin{equation} \label{L^2-met}
	\dim(\mathcal M_z(\alpha,\beta)) \equiv \deg^+(\alpha)-\deg^-(\beta) \mod 8
\end{equation}
Moreover, $\deg^+(\Theta)=-3$, $\deg^-(\Theta)=0$, and $\deg^+(\alpha)=\deg^-(\alpha)$ for an irreducible flat connection $\alpha$ (under the standing assumption that the irreducible flat connections are non-degenerate). The grading on the elements of $\mathcal B(Y)$ defined by $\deg^-$ is called the {\it Floer grading}. Translation along the $\R$ direction gives an action of $\R$ on $\mathcal M_z(\alpha,\beta)$. Unless $\alpha=\beta$ and $z$ is homotopic to the constant path, this action is free and $\breve {\mathcal M}(\alpha,\beta)$ denotes the quotient space with respect to the $\R$-action.

Suppose $\Lambda$ is the following field:
\begin{equation*}
	\Lambda:=\{\sum_i a_i \lambda^{r_i}\mid a_i \in \Q,\, r_i \in \R, \, \lim_{i \to \infty} r_i=\infty\}
\end{equation*}
and $C_*(Y)$ is the $\Lambda$-module generated by non-trivial flat connections on $Y$. There is also an endomorphism of $C_*(Y)$ denoted by $d$ and defined as below:
\begin{equation} \label{boundary}
	d (\alpha):=\sum_{z:\alpha \to \beta} \#\breve {\mathcal M}_z(\alpha,\beta) \cdot \lambda^{\mathcal E(z)}\beta
\end{equation}
where $\alpha$ and $\beta$ are generators of $C_*(Y)$ and $z$ is a path from $\alpha$ to $\beta$. In the above expression, $ \#\breve {\mathcal M}(\alpha,\beta)$ denotes the signed number of the points in $\breve {\mathcal M}(\alpha,\beta)$, with the understanding that this sum is non-zero only if the moduli space is 0-dimensional. In the case that $\breve {\mathcal M}(\alpha,\beta)$ is 0-dimensional, this moduli space is compact and hence it consists of finitely many points. Moreover, for a general path $z$, the moduli space ${\mathcal M}_z(\alpha,\beta)$ is orientable. However, to make \eqref{boundary} rigorous, we need to fix an orientation of ${\mathcal M}_z(\alpha,\beta)$ for each $\alpha$, $\beta$ and $z$. We refer to \cite{Don:YM-Floer} for a consistent way of fixing orientations for these moduli spaces. The flat connection $\beta$ contributes to the sum \eqref{boundary} only if $\deg^-(\alpha)-\deg^-(\beta)=1$. Therefore, the degree of $d$ with respect to the Floer grading is equal to $-1$. The subspace of the elements of degree $i$ in $C_*(Y)$ is denoted by $C_i(Y)$.

The moduli spaces of ASD connections asymptotic to the trivial connection on one end can be utilized to construct two other maps. Following \cite{Don:YM-Floer}, let $D_1:C_*(Y) \to \Lambda$ and $D_2:\Lambda \to C_*(Y)$ be $\Lambda$-homomorphisms defined as follows:
\begin{equation*}
	D_1(\alpha):=\sum_{z:\alpha \to \Theta}\#\breve {\mathcal M}_z(\alpha,\Theta) \cdot \lambda^{\mathcal E(z)}\hspace{1cm}
	D_2(1):=\sum_{z:\Theta \to \beta}\#\breve {\mathcal M}_z(\Theta,\beta) \cdot \lambda^{\mathcal E(z)} \beta
\end{equation*}
From the definition, it is clear that $D_1$ is non-zero only on $C_1(Y)$, and $D_2$ takes values in the summand $C_4(Y)$ of $C_*(Y)$. 

There is another interesting operator $U$ which acts on $C_*(Y)$. The map $U$ is defined with the same formula as \eqref{boundary} except that $\breve{\mathcal M}_z(\alpha,\beta)$ is replaced with $\mathcal N_z(\alpha,\beta)$, a co-dimension 4 submanifold of $\mathcal M_z(\alpha,\beta)$,
\begin{equation} \label{U}
	U (\alpha):=\sum_{z:\alpha \to \beta} -\frac{1}{2}\#\mathcal N_z(\alpha,\beta) 
	\cdot \lambda^{\mathcal E(z)}\beta.
\end{equation}
The factor $-\frac{1}{2}$ does not have a great significance and is included to simplify subsequent formulas. Each homology class $\sigma\in H_i(Y)$ gives rise to a cohomology class $\mu(\sigma)\in H^{4-i}(\mathcal B_z(\alpha,\beta))$ \cite{DK,Don:YM-Floer}. The submanifold $\mathcal N_z(\alpha,\beta)$ of $\mathcal M_z(\alpha,\beta)$ can be regarded as a subspace of $\mathcal M_z(\alpha,\beta)$ representing the cohomology class $\mu(4\cdot x)$ where $x$ is the generator of $H_0(Y)$. To be more specific, let $\mathcal B^*((-1,1)\times Y)$ be the configuration space of irreducible $\SU(2)$-connections on $(-1,1)\times Y$ with finite $L^2_l$ norms. Associated to the base point $(0,y_0)\in (-1,1)\times Y$, there is a base point fibration ${\mathbb E}$ on $\mathcal B^*((-1,1)\times Y)$. We fix two sections $s_1$, $s_2$ of the complexified bundle ${\mathbb E}\otimes \C$ and define $\mathcal N_z(\alpha,\beta)$ as follows:
\begin{equation} \label{N-z}
  \mathcal N_z(\alpha,\beta):=\{[A]\in \mathcal M_z(\alpha,\beta)\mid \text{$s_1(r([A]))$ and $s_2(r([A]))$
  are linearly dependent.}\}
\end{equation}
Here $r:\mathcal M_z(\alpha,\beta)\to \mathcal B^*((-1,1)\times Y)$ is given by the restriction of a connection to $(-1,1)\times Y$ and is well-defined because of unique continuation. We also assume that $s_1$ and $s_2$ are chosen generically such that $\mathcal N_z(\alpha,\beta)$ is cut down transversely. We refer the reader to \cite{Don:YM-Floer} and \cite{Fro:h-inv} for more details\footnote{The corresponding operator in \cite{Don:YM-Floer} is equal to $-\frac{1}{4}U$ and the corresponding operator in \cite{Fro:h-inv} is $-2U$.}. By definition, $U$ has degree $-4$ because $\mathcal N_z(\alpha,\beta)$ is 0-dimensional only if $\deg^-(\alpha)-\deg^-(\beta)=4$.

Let $\widetilde C_*(Y):=C_*(Y)\oplus \Lambda \oplus C_{*-3}(Y)$. Here $C_{*-3}(Y)$  is the same $\Lambda$-module as $C_*(Y)$ whose grading is shifted up by $3$. The summand $\Lambda$ is in correspondence with the trivial connection and we define its Floer grading to be 0. The above maps can be combined to form an operator $\widetilde d:\widetilde C_*(Y) \to \widetilde C_*(Y)$:
\begin{equation}\label{shape-SO-com}
	\left[
	\begin{array}{ccc}
		d&0&0\\
		D_1&0&0\\
		U&D_2&-d 
	\end{array}
	\right]
\end{equation}
\begin{prop}
	The map $\widetilde d$ defines a differential of degree $-1$ on $\widetilde C$.
\end{prop}
\begin{proof}
	The same proposition with the coefficient ring $\Q$ is proved in \cite{Don:YM-Floer}. (See also \cite{Fro:h-inv}.)
	The argument there can be easily adapted to our case where the coefficient ring is $\Lambda$ 
	and the differential is weighted by powers of $\lambda$. 
	We only need to observe that the topological energy of the concatenation of two paths $z_1$ and 
	$z_2$ in $\mathcal B(Y)$ is equal to the sum of the topological energies.
\end{proof}

Define the {\it minimal-degree} of a non-zero element $\eta=\sum_i a_i \lambda^{r_i}$ of $\Lambda$ to be $\min_{i}\{r_i\}$. We write $\mdeg(\eta)$ for this real number. We can also extend the minimal-degree to the case that $\eta=0$ by requiring that $\mdeg(0)=\infty$. The minimal-degree of $(\eta_1,\dots,\eta_k)\in \Lambda^k$ is defined as:
\begin{equation*}
	\mdeg(\eta_1,\dots,\eta_k):=\min_{i}\{\mdeg(\eta_i)\}
\end{equation*}
In particular, $\mdeg$ can be defined on $\widetilde C_*(Y)$.  The differential $\widetilde d$ increases the minimal degree, because the moduli spaces involved in the definition of this differential are non-empty only if the energy of the corresponding path is positive.

One often faces with homology spheres $Y$ that the Chern-Simons functional has degenerate critical points or the moduli spaces $\mathcal M_z(\alpha,\beta)$ are not regular. In these cases, we need to perturb the Chern-Simons functional. The classical choice of such perturbations are given by {\it holonomy perturbations} \cite{Don:ori,Tau:Cas-inv,Fl:I,Don:YM-Floer,K:higher, KM:yaft}. Without going into details of holonomy perturbations, we explain what such perturbations provide for us. Here we follow the approach in \cite{Don:YM-Floer,KM:yaft}. For each 3-manifold $Y$, one can define a family of holonomy perturbations parametrized by a Euclidean space $\mathcal P$: to each $\pi \in \mathcal P$, one can associate a bounded function $f_\pi: \mathcal B(Y) \to \R$ and this association is a linear map from $\mathcal P$ to the set of continuous maps. If $\alpha$ and $\beta$ are two critical points of $\CS+f_\pi$ and $z$ is a path from $\alpha$ to $\beta$, then the moduli space of downward gradient flow lines from $\alpha$ to $\beta$ is denoted by $\mathcal M^\pi_z(\alpha,\beta)$. More precisely, $\mathcal M^\pi_z(\alpha,\beta)$ consists of the gauge equivalence classes of connections $A \in \mathcal A_z(\alpha,\beta)$ such that:
\begin{equation*}
	F^+(A)+(dt\wedge \nabla_{A_t} f_\pi)^+=0
\end{equation*}
where $A_t:=A|_{\{t\}\times Y}$ and $\nabla_{A_t} f_\pi$ is the formal gradient of $f_\pi$ at the connection $A_t$. 
\begin{prop}[\cite{Don:YM-Floer,KM:yaft}] \label{regular-pert}
	For any positive real number $\epsilon$, there is an element $\pi\in \mathcal P$ 
	such that the following properties hold:
	\vspace{-5pt}
	\begin{itemize}
		\item[(i)] $|\pi|_{\mathcal P}< \epsilon$;
		\item [(ii)] the non-trivial critical points of the functional $\CS+f_\pi$ are non-degenerate and irreduicble;
		\item [(iii)] for any two critical points $\alpha$ and $\beta$ of $\CS+f_\pi$ and any path $z$ with $\ind(z)<8$, 
		the moduli space $\mathcal M^\pi_z(\alpha,\beta)$ consists of regular solutions, i.e., 
		any element of this space is cut down transversely;
		\item [(iv)] { if a flat connection is already non-degenerate,
		then we can assume that $f_{\pi}$ vanishes in a neighborhood of this 
		flat connection.}
	\end{itemize}
	
\end{prop}

Fix a positive real number $\epsilon$, and let $f_\pi$ be a perturbation given by Proposition \ref{regular-pert}. The trivial connection is always a critical point of $\CS+f_\pi$, and if $\epsilon$ is small enough, all the other critical points are irreducible. Any such $\pi$ is called an {\it $\epsilon$-admissible} perturbation. We also say a perturbation is admissible if it is an $\epsilon$-admissible perturbation for some $\epsilon$. By replacing $\CS$ with $\CS+f_\pi$, we can construct a chain complex $(\widetilde C_*^\pi(Y), \widetilde d)$ in an analogous way. The chain group can be also equipped with the Floer grading and $\mdeg$. The differential decreases the Floer grading by $-1$. However, the differential $ \widetilde d$ does not necessarily increase the minimal-degree anymore. Nevertheless, there is an upper bound in terms of $\epsilon$ on how much $\widetilde d$ decreases $\mdeg$. If the moduli space $\mathcal M^\pi_z(\alpha,\beta)$ is non-empty, then:
\begin{equation*}
	\mathcal E(z)+f_\pi(\alpha)-f_\pi(\beta)\geq 0.
\end{equation*}
This inequality holds because an element of $\mathcal M^\pi_z(\alpha,\beta)$ can be regarded as a downward gradient flow line of $\CS+f_\pi$. If $\epsilon$ is small enough, we can assume that the $C^0$ norm of $f_{\pi}$ is less than a given positive number $\delta$. Therefore, $\mathcal E(z)\geq-2\delta$, which implies that the differential $\widetilde d$ decreases the minimal-degree on $\widetilde C_*^\pi(Y)$ by at most $2\delta$.

\subsection{Cobordism Maps}\label{tilde-func}
In this subsection, we discuss the functorial properties of instanton Floer homology with respect to cobordisms. Suppose $Y$ , $Y'$ are two integral homology spheres with base points, and $\pi$, $\pi'$ are $\epsilon$-admissible perturbations of the Chern-Simons functional on $Y$, $Y'$, respectively. The associated chain complexes to these perturbations are denoted by $(\widetilde C_*^{\pi}(Y),\widetilde d)$ and $(\widetilde C_*^{\pi'}(Y'),\widetilde d)$. Let $W: Y \to Y'$ be a cobordism with a choice of a path between the base points of $Y$, $Y'$ such that { $b^+(W)=0$, $H_1(W,\Z)=0$}. Then we can define a chain map $\widetilde C_W:\widetilde C_*^{\pi}(Y) \to \widetilde C_*^{\pi'}(Y')$, which has the following matrix form with respect to the standard filtrations of $\widetilde C_*^{\pi}(Y)$ and $\widetilde C_*^{\pi'}(Y')$:
\begin{equation}\label{total-cob-map}
	\left[
	\begin{array}{ccc}
		\varphi(W)&0&0\\
		\Delta_1(W)&{ 1}&0\\
		\mu(W)&\Delta_2(W)&\varphi(W)
	\end{array}
	\right]
\end{equation}
In the case that the choice of the cobordism $W$ is clear from the context, we drop $W$ from our notation and denote the above maps with $\varphi$, $\Delta_1$, $\Delta_2$ and $\mu$.

As in the case of of the differentials of the Floer complexes, these maps are defined by the moduli space of ASD connections. Firstly, we need to fix a Riemannian metric on $W$. We assume that this metric is the product metric in a collar neighborhood of its boundary corresponding to the chosen metrics on $Y$ and $Y'$. Let also $W^+$ be the non-compact manifold that is given by gluing the cylinders $(-\infty,0] \times Y$ and $[0,\infty) \times Y'$ with the product metric to $W$.

Let $\alpha$, $\alpha'$ be respectively generators of $C_*^{\pi}(Y)$, $C_*^{\pi'}(Y')$ and fix connections representing these elements of $\mathcal B(Y)$, $\mathcal B(Y')$. Let $A$ be a connection on the trivial $\SU(2)$-bundle over $W^+$ which is equal to the pull back of the chosen representatives on the cylindrical ends. If $A'$ is another connection with the similar properties, then we say $A$ and $A'$ are equivalent to each other if there is an automorphism $g$ of the trivial $\SU(2)$-bundle such that $g^*(A)-A'$ is compactly supported. An equivalence class of this relation is called a {\it path form $\alpha$ to $\alpha'$ along $W$}. The fundamental group of $\mathcal B(Y)$ (resp. $\mathcal B(Y')$), based at the connection $\alpha$ (resp. $\alpha'$), acts faithfully and transitively  on the space of paths from $\alpha$ to $\alpha'$ by concatenation. The topological energy of a path $z$, represented by a connection $A$, is defined to be:
\begin{equation*}
	\mathcal E(z):=\frac{1}{8\pi^2}\int_{W^+} \tr(F(A)\wedge F(A)).
\end{equation*}
This energy is well-defined and only depends on $z$. Moreover, we have the following generalization of \eqref{energy-CS}:
\begin{equation} \label{energy-CS-cob}
	\mathcal E(z)\equiv \CS(\alpha)-\CS(\alpha') \mod \Z
\end{equation}

For a path $z$ from $\alpha$ to $\alpha'$ represented by a connection $A_0$, define $\mathcal A_z(W; \alpha,\alpha')$ to be the space of $\SU(2)$-connections on $W^+$ such that $|\!|A-A_0|\!|_{L^2_l}<\infty$. Let $\mathcal G_z(W;\alpha,\alpha')$ be the group of the automorphisms $g$ of the trivial $\SU(2)$-bundle such that $|\!|\nabla_{A_0} g|\!|_{L^2_l}<\infty$. The group $\mathcal G_z(W;\alpha,\alpha')$ acts on $\mathcal A_z(W; \alpha,\alpha')$ and the quotient space is denoted by $\mathcal B_z(W; \alpha,\alpha')$. We consider perturbations of the ASD equation on the space $\mathcal B_z(W; \alpha,\alpha')$ which is compatible with the perturbations $\pi$ and $\pi'$. This equation has the following form: 
\begin{equation} \label{ASD-cob}
	F^+(A)+G_0(A)+G_1(A)=0.
\end{equation}
The term $G_0(A)$ is defined using $\pi$ and $\pi'$ in the following way:
\begin{equation*}
	G_0(A):=\phi(t)(dt\wedge \nabla_{A_t} f_{\pi})^++\phi'(t)(dt\wedge\nabla_{A_t} f_{\pi'})^+
\end{equation*}
where $\phi$ (respectively, $\phi'$) is a smooth function on $\R$ that is equal to $1$ on $(-\infty,-2]$ (respectively, $[2,\infty)$) and equal to $0$ on $[-1,0]$ (respectively, $[0,1]$). The functions $\phi$ and $\phi'$ clearly determine functions on $W^+$ which are respectively supported on the incoming end and the outgoing end of $W^+$. The term $G_1(A)$ in \eqref{ASD-cob} is a secondary holonomy perturbation which is supported on a compact subspace of $W^+$. The parametrizing space for this secondary holonomy perturbation is a Euclidean space $\overline {\mathcal P}$. Let $\mathcal M^{\overline \pi}_z(W;\alpha,\alpha')$ be the moduli space of solutions to the equation in \eqref{ASD-cob}.

As in the case of cylinders, we can associate an integer $\ind(z)$ to the path $z$, which is the expected dimension of the moduli space $\mathcal M^{\overline \pi}_z(W;\alpha,\alpha')$. This integer is defined to be the index of the linearization of the equation in \eqref{ASD-cob} modulo the action of the gauge group and it satisfies the following identity:
\begin{equation} \label{dim-mod-8}
	\ind(z) \equiv \deg^+(\alpha)-\deg^-(\alpha') \mod 8
\end{equation}
The maps involved in the cobordism map $\widetilde C_W$ are defined using low dimensional moduli spaces of the form $\mathcal M^{\overline \pi}_z(W;\alpha,\alpha')$ where $\overline \pi$ is chosen such that all these moduli spaces are cut down transversely: 
\begin{prop}[\cite{Don:YM-Floer}]\label{sec-per}
	Given a positive real number $\epsilon$, there is a secondary perturbation term 
	$\overline \pi\in \overline {\mathcal P}$ such that 
	$|\overline \pi|<\epsilon$ and all the moduli space $\mathcal M^{\overline \pi}_z(W;\alpha_0,\alpha_1)$ with 
	${\rm index}(z)<8$ consists of regular solutions, i.e., 
	any element of these spaces is cut down transversely.  
\end{prop}
\noindent 
Any secondary perturbation $\overline \pi$ that satisfies the properties of Proposition \ref{sec-per} is called an {\it $\epsilon$-admissible secondary perturbation}. We say $\overline \pi$ is admissible if it is $\epsilon$-admissible for some choice of $\epsilon$. For an $\epsilon$-admissible perturbation, $\mathcal M^{\overline \pi}_z(W;\alpha,\alpha')$ is a smooth manifold whose dimension is equal to $\ind(z)$.

We start with the definition of the map $\varphi:C_*^{\pi}(Y) \to C_*^{\pi'}(Y')$. We firstly fix an an $\epsilon$-admissible secondary perturbation $\overline \pi$. Let $\alpha$ be a generator of $C_*^{\pi}(Y)$ and define:
\begin{equation} \label{varphi}
	\varphi(\alpha):=\sum_{z:\alpha\to\alpha' } \#{\mathcal M^{\overline \pi}_z}(W;\alpha,\alpha') 
	\cdot \lambda^{\mathcal E(z)}\alpha'
\end{equation}
where the sum is over all paths $z$ that the moduli space ${\mathcal M_z^{\overline \pi}}(W;\alpha,\alpha')$ is 0-dimensional. In fact, for each $\alpha'$ there is at most one path $z$ from $\alpha$ to $\alpha'$ such that the moduli space $\mathcal M^{\overline \pi}_z(W;\alpha,\alpha')$ is 0-dimensional and for this choice of $z$, the moduli space is compact. Therefore, it consists of finitely many points. As in the case of differentials, there is a canonical choice of orientation for this moduli space \cite{Don:YM-Floer} and for this choice $\varphi$ is a chain map. 

The other terms in \eqref{total-cob-map} are defined in an analogous way. For example, $\mu:C_*^{\pi}(Y) \to C_{*-3}^{\pi}(Y)$ is defined similar to the map $\varphi$ in \eqref{varphi} by replacing $\#{\mathcal M^{\overline \pi}_z}(W;\alpha,\alpha')$ with $-\frac{1}{2}\#{\mathcal N^{\overline \pi}_z}(W;\alpha,\alpha')$ where $\mathcal N^{\overline \pi}_z(W;\alpha,\alpha')$ is a codimension $3$ submanifold of $\mathcal M^{\overline \pi}_z(W;\alpha,\alpha')$. The definition of $\mathcal N^{\overline \pi}_z(W;\alpha,\alpha')$ is similar to $\mathcal N_z(\alpha,\beta)$ and uses the path between base points. We refer the reader to \cite[Theorem 6]{Fro:h-inv} for the definition of $\mu$. The maps $\Delta_1(W):C_*^{\pi}(Y) \to \Lambda$ and $\Delta_2(W): \Lambda \to C_*^{\pi'}(Y')$ are also defined by considering the moduli spaces  ${\mathcal M^{\overline \pi}_z}(W;\alpha,\alpha')$ where either $\alpha$ or $\alpha'$ is equal to the trivial connection. A standard argument using 1-dimensional moduli spaces over $W^+$ verifies the following proposition:
\begin{prop} \label{CW-chain-map}
	$\widetilde C_W:\widetilde C_*^{\pi}(Y) \to \widetilde C_*^{\pi'}(Y')$ is a chain map.
\end{prop}

The chain map $\widetilde C_W$ behaves well with respect to the Floer grading and $\mdeg$. Identity \eqref{dim-mod-8} implies that the map $\widetilde C_W$ preserves the Floer grading. In order to study the behavior with respect to $\mdeg$, let $\mathcal M^{\overline \pi}_z(W;\alpha,\alpha')$ be a non-empty moduli space that contains the class represented by a connection $A$. Then $\mathcal E(z)$ can be written as the sum of three terms $\mathcal E_0(z)$, $\mathcal E_1(z)$ and $\mathcal E_2(z)$, which are defined in the following way:
\begin{equation*}
	\mathcal E_0(z):=\frac{1}{8\pi^2}\int_{(-\infty,-T]\times Y_0} \tr(F(A)\wedge F(A))\hspace{1cm}
	\mathcal E_1(z):=\frac{1}{8\pi^2}\int_{[T,\infty)\times Y_1} \tr(F(A)\wedge F(A))
\end{equation*}
\begin{equation*}
	\mathcal E_3(z):=\frac{1}{8\pi^2}\int_{W^c} \tr(F(A)\wedge F(A))
\end{equation*}
where $W^c$ is the complement of $(-\infty,-T]\times Y_0$ and $[T,\infty)\times Y_1$ in $W^+$, and $T$ is chosen such that the secondary perturbation term is supported on $W^c$. We can argue as in the previous subsection that for any positive constant $\delta$ there is $\epsilon$ such that $\mathcal E_0(z), \mathcal E_1(z)\geq -2\delta$. We also have:
\begin{equation*}
	\mathcal E_3(z)=\frac{1}{8\pi^2}\int_{W^c} |F^-(A)|^2-|F^+(A)|^2\geq -\frac{1}{8\pi^2}\int_{W^c} |G_0(A)+G_1(A)|^2
\end{equation*}
Therefore, if $\overline \pi$ is given by Proposition \ref{sec-per} and $\epsilon$ is small enough, then we can conclude that $\mathcal E_3(z)\geq -\delta$. Consequently, we can ensure that $\widetilde C_W$ does not decrease $\mdeg$ by more than a given positive number once $\epsilon$ is small enough.

We can relax the condition $H_1(W,\Z)$ to $b_1(W)=0$ using the ideas of \cite{Don:ori}. In the case that $b_1(W)=0$, we have two additional types of reducible flat connections on $W^+$:
\vspace{-5pt}
\begin{enumerate}
	\item[(i)] flat connections induced by representations of $\pi_1(W)$ into $\Z/2\Z$. The stabilizer of any such reducible 		
	connection is a copy of $\SO(3)$;
	\item[(ii)] flat connections induced by representations $\rho$ of $\pi_1(W)$ such that there are elements
	in the image of $\rho$ which does not have order $2$. The stabilizer of any such reducible connection is a copy of 
	$S^1$. The flat connections obtained from $\rho$ and $\rho^{-1}$ are equivalent to each other.
\end{enumerate}
Using holonomy perturbations we may assume that any reducible flat connection of type (i) is regular \cite{Don:ori}. Moreover, we may assume that the linearization of the perturbed ASD equation at any reducible connection of type (ii) has 1-dimensional co-kernel \cite{Don:ori}. 

The presence of flat connections of types (i) and (ii) requires us to modify the proof of Proposition \ref{CW-chain-map}. More specifically, we have the following new equations: 
\begin{equation}\label{two-equations-cW}
	D_1'\circ \lambda= c(W)\cdot D_1+\Delta_1 \circ d,\hspace{1cm}\lambda\circ D_2= c(W)\cdot D_2'-d'\circ \Delta_2,
\end{equation}	
where $c(W)$ is the number of the elements of $H_1(W,\Z)$, which is a finite positive integer. The first equation in \eqref{two-equations-cW} is  obtained by looking at the 1-dimensional moduli space associated to a path $z$ along $W$ from a generator of $C_*^{\pi}(Y)$ to the trivial connection $\Theta'$ on $Y'$. The first term on the right hand side of the first equation in \eqref{two-equations-cW} is obtained by counting the ends associated to gluing an element of the 0-dimensional moduli space $\breve{\mathcal M}^{\pi}_z(Y\times \R,\alpha,\Theta)$ to flat connections of type (i) and (ii). Using the arguments of \cite{Don:ori}, any flat connection of type (i) contributes one end to the glued up moduli space and any flat connection of type (ii) contributes two ends. Therefore, in total we have the coefficient $c(W)$ in our formula. A similar explanation applies to the second equation in \eqref{two-equations-cW}. In summary, Proposition \ref{CW-chain-map} holds if we replace \eqref{total-cob-map} with \eqref{total-cob-map-mod}:
\begin{equation}\label{total-cob-map-mod}
	\left[
	\begin{array}{ccc}
		\varphi(W)&0&0\\
		\Delta_1(W)&{ c(W)}&0\\
		\mu(W)&\Delta_2(W)&\varphi(W)
	\end{array}
	\right]
\end{equation}

\subsection{Equivariant Instanton Floer Homology Groups}\label{equiv}
In this section, we review the definition of three different {\it equivariant} Floer homology groups $\hrI_*(Y)$, $\crI_*(Y)$ and $\brI_*(Y)$ for any integral homology sphere $Y$. These Floer homology groups are $\Lambda[x]$-modules and are constructed algebraically from the Floer chain complexes of Subsection \ref{C-tilde} without any other geometrical input. The Floer homology groups $\hrI_*(Y)$ and $\crI_*(Y)$ are essentially the same as the Floer homology groups $\overline{\overline{\rm HF}}(Y)$, $\underline{\underline{\rm HF}}(Y)$ introduced in \cite{Don:YM-Floer}. The notation here is motivated by the notations used in the context of monopole Floer homology \cite{KM:monopoles-3-man}. Similar to the three flavors of monopole Floer homology, the Floer homology groups $\hrI(Y)$, $\crI(Y)$ and $\brI(Y)$ are modules over $\Lambda[x]$ and they fit into an exact sequence of the form:
\begin{equation}\label{ftb-exact}
	\xymatrix{
	\crI_*(Y) \ar[rr]^{j_*}& &
	 \hrI_*(Y) \ar[dl]^{p_*}\\
	& \brI_*(Y) \ar[ul]^{i_*} &
	}
\end{equation}

Following Subsection \ref{C-tilde}, we fix an $\epsilon$-admissible perturbation $\pi$ for the Chern-Simons functional of $Y$ and form the chain complex  $(C_*^{\pi}(Y),d)$ and the maps $U$, $D_1$ and $D_2$. The Floer homology group $\hrI_*(Y)$ (pronounced as ``I-from'') is the homology of the following chain complex:
\[
  \hrC_*^{\pi}(Y):=C_*^{\pi}(Y)\oplus \Lambda[x]\hspace{1cm}\widehat d(\alpha,\sum_{i=0}^{N}a_ix^i)=(d\alpha-\sum_{i=0}^NU^iD_2(a_i),0)
\]
The $\Lambda[x]$-module structure is also induced by the map:
\[
  x\cdot(\alpha,\sum_{i=0}^{N}a_ix^i):=(U\alpha,D_1(\alpha)+\sum_{i=0}^{N}a_ix^{i+1})
\]
It is helpful to think about the summand $\Lambda[x]$ in $\hrC_*^{\pi}(Y)$ as a free $\Lambda[x]$-module generated by the trivial connection. Similar comments apply to $\crI_*(Y)$ and $\brI_*(Y)$ described below.

The Floer homology group $\crI_*(Y)$ (pronounced as ``I-to'') is the homology of the following chain complex:
\[
  \crC_*^{\pi}(Y):=C_*^{\pi}(Y)\oplus \Lambda[\![x^{-1},x]/\Lambda[x]\hspace{1cm}\widecheck d(\alpha,\sum_{i=-\infty}^{-1}a_ix^{i})=(d\alpha,\sum_{i=-\infty}^{-1}D_1U^{-i-1}(\alpha)x^{i})
\]
Here $\Lambda[\![x^{-1},x]$ is the ring of Laurent power series in $x^{-1}$. The action of $x\in \Lambda[x]$ on this vector space over $\Lambda$ is given by multiplication by $x$. The quotient module $\Lambda[\![x^{-1},x]/\Lambda[x]$ can be identified in an obvious way with the set of power series in $x^{-1}$ with vanishing constant terms. The $\Lambda[x]$-module structure on $\crC_*^{\pi}(Y)$ is given by the map:
\[
  x\cdot(\alpha,\sum_{i=-\infty}^{-1}a_ix^{i}):=(U\alpha+D_2(a_{-1}),\sum_{i=-\infty}^{-2}a_ix^{i+1})
\]
Similar to the monopole Floer homology group $\overline{HM}(Y)$, the Floer homology group $\brI_*(Y)$ (pronounced as ``I-bar'') has a very simple form. This $\Lambda[x]$-module is defined to be $\Lambda[\![x^{-1},x]$. 

Next, we define chain maps:
\begin{equation}\label{surg-triangle-top}
	\xymatrix{
	\crC_*^{\pi}(Y) \ar[rr]^{j}& &
	 \hrC_*^{\pi}(Y) \ar[dl]^{p}\\
	& \brC_*(Y) \ar[ul]^{i} &
	}
\end{equation}
by the formulas:
\begin{equation}\label{i}
  i(\sum_{i=-\infty}^{N}a_ix^i)=(\sum_{i=0}^{N}U^{i}D_2(a_i),\sum_{i=-\infty}^{-1}a_ix^{i}),
\end{equation}
\begin{equation}\label{j}
 j(\alpha,\sum_{i=-\infty}^{-1}a_ix^{i})=(\alpha,0),
\end{equation}
\begin{equation}\label{p}
  p(\alpha,\sum_{i=0}^{N}a_ix^i)=\sum_{i=-\infty}^{-1}D_1U^{-i-1}(\alpha)x^{i}+\sum_{i=0}^{N}a_ix^i.
\end{equation}
\begin{lemma}
	The maps $i$ and $p$ are $\Lambda[x]$-module homomorphisms which are also chain maps, and $j$ is a chain map which commutes with the action
	of $x$ up to a chain homotopy. 
	That is to say, there is a map $h:\crC_*^{\pi}(Y) \to \hrC_*^{\pi}(Y)$ such that:
	\begin{equation}\label{commute-with-x}
	  i\circ x=x\circ i,\hspace{1cm}j\circ x-x\circ j=\widehat  d \circ h+h\circ  \widecheck d,\hspace{1cm}p\circ x=x\circ p.
	\end{equation}
	In particular, if $i_*$, $j_*$ and $p_*$ are the maps induced by $i$, $j$ and $p$ at the level of homology, then
	they are $\Lambda[x]$-module homomorphisms.
	Moreover, $i_*$, $j_*$ and $p_*$ form an exact triangle as in \eqref{ftb-exact}.
\end{lemma}
\begin{proof}
	It is straightforward to check that the first and the third identifies in \eqref{commute-with-x} hold. If we define:
	\[
	  h(\alpha,\sum_{i=-\infty}^{-1}a_ix^{i}):=(0,-a_{-1}).
	\]
	then we obtain the second identity in \eqref{commute-with-x}.	
	We define maps $\frak k$, $\frak l$ and $\frak r$ as in the diagram:
        \begin{equation}
        	\label{surg-triangle}
        	\xymatrix{
        	\crC_*^{\pi}(Y)\ar[dr]_{\frak k} & &
        	 \hrC_*^{\pi}(Y) \ar[ll]_{\frak l}\\
        	& \brC_*(Y)  \ar[ur]^{\frak r} &
        	}
        \end{equation}
	using the formulas:
	\begin{equation}\label{frak-k-l}
 		\frak k(\alpha,\sum_{i=-\infty}^{-1}a_ix^{i})=-\sum_{i=-\infty}^{-1}a_ix^{i},\hspace{1cm}
		\frak l(\alpha,\sum_{i=0}^{N}a_ix^i)=((-1)^{|\alpha|}\alpha,0),
	\end{equation}
	and:
	\begin{equation}\label{frak-l}
 		\frak r(\sum_{i=-\infty}^{N}a_ix^i)=(0,\sum_{i=0}^{N}a_ix^i).
	\end{equation}
	In the definition of $\frak l$, the term $|\alpha|$ denotes the $\Z/8$-grading of $\alpha$.
	In particular, we assume that $\alpha$ is a homogenous element.
	These maps satisfy the identities:
	\begin{equation}
		p\circ j+\frak k \circ \widecheck d=0\hspace{1cm}
		i\circ p+\frak l\circ \widehat d+\widecheck d\circ \frak l=0\hspace{1cm}
		j\circ i+ \widehat d \circ \frak r=0
	\end{equation}
	It is also straightforward to check that the following maps are respectively isomorphisms of the chain 
	complexes $\crC_*^{\pi}(Y)$, $\hrC_*^{\pi}(Y)$ and $\brC_*(Y)$:
	\[
	  \frak l\circ j+i\circ \frak k\hspace{1cm}\frak r\circ p+j\circ \frak l\hspace{1cm}\frak k\circ i+p\circ \frak r
	\]
	This implies that \eqref{surg-triangle-top} determines an exact triangle at the level of homology groups. 
	(See \cite[Lemma 4.2]{OzSz:HF-branch} and \cite[Lemma 7.1]{KM:Kh-unknot}.)
\end{proof}

Equivariant instanton Floer homology groups and the exact triangle in \eqref{ftb-exact} are functorial with respect to cobordisms. Suppose $Y$ and $Y'$ are two integral homology spheres and $W:Y\to Y'$ is a cobordism with { $b_1(W)=b^+(W)=0$}. Suppose $\epsilon$-admissible perturbations $\pi$ and $\pi'$ of the Chern-Simons functionals of $Y$ and $Y'$ are fixed, and these perturbations are extended to an $\epsilon$-admissible perturbation of the ASD equation on $W^+$. As in the previous subsection, we can associate the maps $\varphi$, $\mu$, $\Delta_1$ and $\Delta_2$ to $W$ by the chosen perturbations. We use these maps as the only geometrical input to obtain the functoriality of the exact triangle in \eqref{ftb-exact}. We firstly define a homomorphism $\hrC_W:\hrC_*^{\pi}(Y) \to \hrC_*^{\pi'}(Y')$ as follows:
\begingroup
\small
\begin{align*}
  \hrC_W(\alpha,&\sum_{i=0}^{N}a_ix^i)=\left(\varphi(\alpha)+\sum_{i=0}^N(U')^i\Delta_2(a_i)+
  \sum_{i=0}^{N}\sum_{k=0}^{i-1}(U')^k\mu U^{i-1-k}D_2(a_i),{ c(W)\cdot}\sum_{i=0}^Na_ix^i\right.\\
  &+\sum_{i=0}^{N-1}(\sum_{k=i+1}^{N}D_1'(U')^{k-i-1}\Delta_2(a_k))x^i+\sum_{i=0}^{N-1}(\sum_{k=i+1}^{N}\Delta_1 U^{k-i-1}D_2(a_k))x^i\\
  &\left.+\sum_{i=0}^{N-2}(\sum_{j=i+1}^{N-1}\sum_{k=j+1}^{N}
  D_1'(U')^{j-i-1}\mu U^{k-j-1}D_2(a_k))x^i\right)
\end{align*}
\endgroup
{ Here $c(W)$ denotes the number of the elements of $H_1(W,\Z)$.} Similarly, $\crC_W:\crC_*^{\pi}(Y) \to \crC_*^{\pi'}(Y')$ is defined to be:
\begingroup
\small
\begin{align*}
  \crC_W(\alpha&,\sum_{i=-\infty}^{-1}a_ix^i)=\left(\varphi(\alpha), \sum_{i=-\infty}^{-1}\Delta_1U^{-i-1}(\alpha)x^i+
  \sum_{i=-\infty}^{-2}\sum_{j=i+1}^{-1}D_1'(U')^{-j-1}\mu U^{j-i-1}(\alpha)x^{i}\right.\\
  &+(\sum_{i=-\infty}^{-1}a_ix^i)\cdot({ c(W)}+\sum_{j=-\infty}^{-1}\Delta_1U^{-j-1}D_2(1)x^{j}+\sum_{j=-\infty}^{-1}D_1'(U')^{-j-1}\Delta_2(1)x^{j}\\
 & \left.+\sum_{j=-\infty}^{-1}\sum_{k=-\infty}^{-1}D_1'(U')^{-j-1}\mu U^{-k-1}D_2(1)x^{j+k})\right)
\end{align*}
\endgroup
Finally, we define $\brC_W:\brC_*^{\pi}(Y) \to \brC_*^{\pi'}(Y')$:
\begingroup
\small
\begin{align*}
  \brC_W(\sum_{i=-\infty}^{N}a_ix^i)=&(\sum_{i=-\infty}^{N}a_ix^i)\left({ c(W)}+\sum_{i=-\infty}^{-1}\Delta_1U^{-i-1}D_2(1)x^{i}
  +\sum_{i=-\infty}^{-1}D_1'(U')^{-i-1}\Delta_2(1)x^{i}\right.\\
  &\left. +\sum_{k=-\infty}^{-1}\sum_{j=-\infty}^{-1}D_1'(U')^{-j-1}\mu U^{-k-1}D_2(1)x^{k+j}\right)
\end{align*}
\endgroup
\begin{prop}\label{func-chain-map}
	The maps $\hrC_W$, $\crC_W$ and $\brC_W$	 are chain maps.
	The map $\brC_W$ is a $\Lambda[x]$-module homomorphism, and $\hrC_W$ and $\crC_W$ commute 
	with the action of $x$ up to chain homotopies. That is to say, there are maps $\fK:$ and $\fL$ such that:
	\[
	  x\circ\hrC_W-\hrC_W\circ x=\fK\circ \widehat d+\widehat d' \circ \fK \hspace{1cm}
	  x\circ \crC_W-\crC_W\circ x=\fL\circ \widecheck d+\widecheck d' \circ \fL
	\]
	Moreover, there are maps
	$K:\hrC_*^{\pi}(Y)\to \brC_*^{\pi}(Y')$ and $L:\brC_*^{\pi}(Y)\to \crC_*^{\pi}(Y')$ such that:
	\[
	  p'\circ \hrC_W-\brC_W\circ p=K \circ \widehat d\hspace{1cm}j'\circ\crC_W=\hrC_W\circ j\hspace{1cm}
	  i'\circ\brC_W-\crC_W\circ i=\widecheck d' \circ L
	\]
	In particular, $\hrC_W$, $\crC_W$ and $\brC_W$ induce a $\Lambda[x]$-module homomorphism of 
	exact triangles at the level of homology.
\end{prop}

\begin{proof}
	The first part is easy to verify. For the remaining parts, we can define $\fK$, $\fL$, $K$ and $L$ as follows:
	\begin{equation}\label{fKfL}
		\fK(\alpha,\sum_{i=0}^Na_ix^i)=(\mu(\alpha),\Delta_1(\alpha)),\hspace{1cm}
		\fL(\alpha,\sum_{i=-\infty}^{-1}a_ix^i)=(\mu(\alpha)+\Delta_2(a_{-1}),0),
	\end{equation}	
	\begin{equation}\label{K}
	K(\alpha,\sum_{i=0}^Na_ix^i)=\sum_{i=-\infty}^{-1}\Delta_1U^{-i-1}(\alpha)x^{i}
	+\sum_{k=-\infty}^{-1}\sum_{j=-\infty}^{-1}D_1'(U')^{-j-1}\mu U^{-k-1}(\alpha)x^{k+j},
	\end{equation}
	\begin{equation}\label{L}
		L(\sum_{i=-\infty}^{N}a_ix^i)=(\sum_{i=0}^{N}(U')^{i}\Delta_2(a_i)+\sum_{i=0}^{N}\sum_{j=0}^{i-1}(U')^{j}\mu U^{i-1-j} D_2(a_i),0).
	\end{equation}	
\end{proof}

\begin{remark} \label{aux-choice-comp}
	Suppose $W:Y\to Y'$ is a { negative definite cobordism with trivial 
	$b_1$}. Standard continuation maps can be used to show that if one changes the choices of the perturbation
	term $\overline \pi$ and the Riemannian metric on $Y$, then $\widetilde C_W$ changes by a chain homotopy 
	of the following form:
	 \begin{equation}\label{total-cob-hom}
            	\left[
            	\begin{array}{ccc}
            		\psi&0&0\\
            		K_1&0&0\\
            		\nu&K_2&-\psi
            	\end{array}
            	\right].
	\end{equation}
	This input can be used to show that the chain homotopy type of the maps $\hrC_W$, $\crC_W$ and 
	$\brC_W$ are also independent of the perturbation term and the metric. 
	
	Suppose $W':Y'\to Y''$ is another { negative definite cobordism with trivial 
	$b_1$.} Then we can form the composite cobordism $W\#W':Y\to Y''$.
	Standard neck stretching argument shows that the map $\widetilde C_{W\#W'}$ is chain homotopic to 
	$\widetilde C_{W'} \circ \widetilde C_W$. Using this input, it is straightforward (but slightly cumbersome) 
	to show that $\widehat C_{W\#W'}$ 
	is chain homotopic to $\widehat C_{W'} \circ \widehat C_W$. Similar results hold for  
	$\widecheck C_{W\#W'}$ and $\widebar C_{W\#W'}$. We do not attempt to write down these chain 
	homotopies here partly because we do not need them. In fact, there are larger {\it models} for 
	the homology groups $\hrI_*(Y)$, $\crI_*(Y)$ and 
	$\brI_*(Y)$ where these results are easier to verify. In particular, these alternative models will 
	behave better with respect to connected sum of integral homology spheres \cite{D:conn-Gamma}. 
\end{remark}

We can extend the definition of the grading $\mdeg$ to the non-zero elements of the complexes $\hrC_*^{\pi}(Y)$, $\crC_*^{\pi}(Y)$ and $\brC_*^{\pi}(Y)$:
\[
  \mdeg(\alpha,\sum_{i=0}^{N}a_ix^i)=\left\{
  \begin{array}{ll}
	 \mdeg(a_0,a_1,\dots,a_N)& \text{if $\sum_{i=0}^{N}a_ix^i\neq 0$},\\
	  \mdeg(\alpha)& \text{if $\sum_{i=0}^{N}a_ix^i=0$,}
  \end{array}
  \right.
\]
\[
    \mdeg(\alpha,\sum_{i=-\infty}^{-N}a_ix^i)=\left\{
  \begin{array}{ll}
	 \mdeg(\alpha)& \text{if $\alpha\neq 0$,}\\
	 \mdeg(a_{-N})& \text{if $\alpha=0$ and $a_{-N}\neq 0$,}
  \end{array}
  \right.  
\]
and
\[
  \mdeg(\sum_{i=-\infty}^{N}a_ix^i)=\left\{
  \begin{array}{ll}
	 \mdeg(a_0,a_1,\dots,a_N)& \text{if $\sum_{i=0}^{N}a_ix^i\neq 0$},\\
	  \mdeg(a_N)& \text{if $N<0$.}
  \end{array}
  \right.
\]
We define $\mdeg$ of the trivial element in these three complexes to be $\infty$. It is also useful to pick a notation for the standard notion of degree for non-zero elements of $\brC_*^{\pi}(Y)=\Lambda[\![x^{-1},x]$:
\[
  \hspace{2.5cm}{\rm Deg}(\sum_{i=-\infty}^Na_ix^i)=N\hspace{1cm}\text{if $a_N\neq 0$.}
\]
The following lemma is a straightforward consequence of our analysis of the previous subsection on the behavior of cobordism maps with respect to $\mdeg$:
\begin{lemma}\label{mdeg-func}
	For any positive real number $\delta$, there is a positive constant $\epsilon$ such 
	that if the perturbations $\pi$, $\pi'$ and the secondary perturbation on $W$ are $\epsilon$-admissible, then 
	the maps $\hrC_W$ and $\crC_W$ do not decrease $\mdeg$ by more than $\delta$. 
	Moreover, for any $z\in \brC_*^{\pi}(Y)$,
	the difference $|\mdeg(\brC_W(z))-\mdeg(z)|$ is at most $\delta$ and ${\rm Deg}(\brC_W(z))={\rm Deg}(z)$.
\end{lemma}
\noindent
Note that the claim ${\rm Deg}(\brC_W(z))={\rm Deg}(z)$ holds for any choices of admissible perturbations and we do not need any assumption on $\epsilon$.

We end this subsection with some speculations about equivariant instanton Floer homology groups. An immediate consequence of the claimed properties in Remark \ref{aux-choice-comp} is that the exact triangle \eqref{ftb-exact} is a topological invariant of $Y$ and does not depend on the choices of the perturbation term and the Riemannian metric on $Y$. By dropping the powers of $\lambda$ in our definitions, we can similarly define the analogue of the exact triangle \eqref{ftb-exact} with rational coefficients.
	
The three flavors of monopole Floer homology $\widehat {\rm HM}(Y)$, $\widecheck {\rm HM}(Y)$ and $\overline {\rm HM}(Y)$ are also 3-manifold invariants which fit into an exact triangle of $\Q[x]$-modules\footnote{In the context of monopole Floer homology, $\Q[x]$ should be regarded as the cohomology ring of ${\rm BS^1}$. On the other hand, equivariant instanton Floer homologies (with rational coefficients) are modules over the cohomology ring of ${\rm BSO}(3)$, which is again isomorphic to $\Q[x]$.}\cite{KM:monopoles-3-man}. Moreover, for integral homology spheres, the invariant $\overline {\rm HM}(Y)$ is always isomorphic to $\Q[x^{-1},x]\!]={\rm Hom}(\Q[\![x^{-1},x],\Q)$. These similarities motivate the following question. An affirmative answer to this question would be in the sprit of Witten's conjecture relating Donaldson invariants and Seiberg-Witten invariants \cite{Wit:SW,GNY:D=W,FL:D=W}. 
\begin{question}\label{Witten-3man}
	Is there any relation between the instanton invariants $\hrI(Y)$, $\crI(Y)$ 
	and the monopole invariants $\widehat {\rm HM}(Y)$, $\widecheck {\rm HM}(Y)$?
	How about the exact triangle \eqref{ftb-exact} and the corresponding one for 
	monopole Floer homology groups 
	$\widehat {\rm HM}(Y)$, $\widecheck {\rm HM}(Y)$ and $\overline {\rm HM}(Y)$?
\end{question}
	
As it is pointed out in Remark \ref{ins-h-mon-h}, Fr\'oyshov's instanton $h$-invariant can be reformulated using the exact triangle \eqref{ftb-exact}. This definition of $h$-invariant is similar to the definition of Monopole $h$-invariant \cite{Fr:SW-4-man-bdry,KM:monopoles-3-man,Fro:mon-h-inv}. As a follow up to Question \ref{Witten-3man}, one can ask:
\begin{question}\label{Witten-3man}
	Is there any relation between instanton and monopole $h$-invariants?	
\end{question}

\section{Homology Cobordism Invariants}
\subsection{Definition of $\Gamma_Y$} \label{GammaY}

We are ready to give the definition of $\Gamma_Y$. For any $k$, define:
\begin{equation} \label{Gamma0}
	\Gamma_Y(k):=\max(\left( \lim_{|\pi|_{\mathcal P}\to 0} \inf_{\substack{z\in\brC_*^{\pi}(Y),\,w\in\crC_*^{\pi}(Y), \\\widecheck d(w)=i(z), \,	{\rm Deg}(z)=-k}}(\mdeg(z)-\mdeg(w)) \right),0)
\end{equation}
Here the limit is taken over a sequence of perturbations $\{\pi_i\}$ where $\pi_i$ is $\epsilon_i$-admissible and $\epsilon_i$ converges to $0$. We use the convention that the infimum of an empty set is equal to $\infty$. In particular, if there is no $z\in \brC_*^{\pi}(Y)$ with $i(z)=0$ and ${\rm Deg} (z)=-k$, then the infimum in the above expression is equal to $\infty$. In the following proposition, we show that the definition of $\Gamma_Y(k)$ is independent of the choice of the sequence of perturbations $\{\pi_i\}$.

\begin{prop}\label{hom-inv-Gamma}
	For any integer $k$, $\Gamma_Y(k)$ depends only on $k$ and the homology cobordism class of $Y$. 
\end{prop}
\begin{proof}
	Suppose $Y$ and $Y'$ are integral homology spheres and $W:Y\to Y'$ is a homology cobordism. 
	Suppose $\pi$ and $\pi'$ are $\epsilon$-admissible perturbations for the Chern-Simons functional of  
	$Y$ and $Y'$. We extend these perturbations to $W^+$ using a secondary $\epsilon$-admissible perturbations. 
	Thus we can associate the 
	chain complexes $\hrC^{\pi}_*(Y)$, $\crC^{\pi}_*(Y)$, $\brC^{\pi}_*(Y)$ to $Y$, 
	the chain complexes $\hrC^{\pi'}_*(Y')$, $\crC^{\pi'}_*(Y')$, 
	$\brC^{\pi'}_*(Y')$ to $Y'$ and the chain maps $\hrC_W$, $\crC_W$ and $\brC_W$ to $W$. 
	
	Let $z\in \brC_*^{\pi}(Y)$ and $w\in\crC_*^{\pi}(Y)$ be chosen such that ${\rm Deg}(w)=-k$ and 
	$\widecheck d(w)=i(z)$. Let $w'=\crC_W(w)$ and $z'=\brC_W(z)$. 
	Proposition \ref{func-chain-map} and Lemma \ref{mdeg-func} assert that:
	\[
	  {\rm Deg}(z')=-k,\hspace{1cm} i(z')=\widecheck d(w'+L(z)).
	\]
	where $L$ is defined in \eqref{L}. We fix a positive constant $\delta$. 
	Using Lemma \ref{mdeg-func}, we can conclude that there is a positive constant
	$\epsilon_0$ such that if $\epsilon\leq \epsilon_0$, then:
	\begin{equation}\label{1st-ineq}
	  \mdeg(w')\geq \mdeg(w)-\delta,\hspace{1cm} |\mdeg(z')-\mdeg(z)|\leq \delta.
	\end{equation}
	We can apply a similar argument  to show that if $\epsilon_0$ is small enough, then:
	\begin{equation}\label{2nd-ineq}
	  \mdeg(L(z))\geq \mdeg(z)-\delta.
	\end{equation}
		
	Identities \eqref{1st-ineq} and \eqref{2nd-ineq} imply that:
	\begin{align*}
	  \mdeg(w'+L(z))&\geq \min (\mdeg(w'), \mdeg(L(z)))\\
	  &\geq \min (\mdeg(w), \mdeg(z))-\delta
	\end{align*}
	Therefore, we have:
	\begin{align*}
	  \mdeg(z')-\mdeg(w'+L(z))&\leq \mdeg(z)- \min (\mdeg(w), \mdeg(z))+2\delta\\
	  &\leq \max(\mdeg(z)-\mdeg(w),0)+2\delta
	\end{align*}
	By taking infimum over all pairs of $(w,z)$ as above we have:
	\begin{align}
	  \inf_{\substack{z'\in\brC_*^{\pi'}(Y'),\,w'\in\crC_*^{\pi'}(Y'), \\
	  \widecheck d(w')=i(z'), \,{\rm Deg}(z')=-k}}(\mdeg(z')-&\mdeg(w'))
	  \leq\nonumber \\ 
	  &\max( \inf_{\substack{z\in\brC_*^{\pi}(Y),\,w\in\crC_*^{\pi}(Y), \\\widecheck d(w)=i(z), 
	  \,{\rm Deg}(z)=-k}}(\mdeg(z)-\mdeg(w)) ,0)+2\delta \label{ineq-cob}
	\end{align}
	By reversing the cobordism $W$, we obtain a similar inequality where the roles of $Y$ and $Y'$ are reversed. 
	Thus the limit in \eqref{Gamma0} 
	converges to a finite number for any sequence of holonomy perturbations 
	$\{\pi_i\}$ with $|\pi_i|_{\mathcal P}$ being convergent to zero. 
	Moreover, this limit is independent of the chosen sequence. 
	In fact, the limit only depends on the homology cobordism class of $Y$ and the integer $k$.
\end{proof}

Next, we attempt to unravel the definition of $\Gamma_Y$. Fix a positive integer $k$, and let $\alpha \in C_*^\pi(Y)$ be chosen such that:
\begin{equation}\label{alpha-pos}
  d \alpha=0, \hspace{1cm}D_1U^{k-1}(\alpha)\neq 0, \hspace{1cm} D_1U^{j}(\alpha)=0  \hspace{.3cm}\text{ for any $j< k-1$.}
\end{equation}
We can form a pair:
\begin{equation} \label{pair}
  z=\sum_{i=-\infty}^{-k}U^{-i-1} D_1(\alpha)x^i\in\brC_*^{\pi}(Y) \hspace{1cm} w=(\alpha,\sum_{i=-\infty}^{-1}a_ix_i)\in \crC_*^{\pi}(Y)
\end{equation}
where the constants $a_i\in \Lambda$ are chosen arbitrarily. Then $z$ has degree $-k$ and $i(z)=\widecheck d(w)$. In fact, any pair of $z$ with degree $-k$ and $w$, satisfying $\widecheck d(w)=i(z)$, are given as in \eqref{pair} for an appropriate choice of $\alpha$ and $\{a_i\}$. In particular, the definition of $\Gamma_Y$ at a positive integer can be rewritten as:
\begin{equation} \label{Gamma1}
	\Gamma_Y(k)=\lim_{|\pi|_{\mathcal P}\to 0}\inf_{\alpha}(\mdeg(D^1U^{(k-1)}(\alpha))-\mdeg(\alpha)))
\end{equation}
where the infimum is taken over all $\alpha \in C_*^\pi(Y)$ that satisfy \eqref{alpha-pos}. Given any positive constant $\delta$, there is $\epsilon$ such that if $\pi$ is an $\epsilon$-admissible perturbation, then the difference $\mdeg(D^1U^{(k-1)}(\alpha))-\mdeg(\alpha))$ is greater than $-\delta$. This is the reason that we do not need to take the maximum of the expression in \eqref{Gamma1} and $0$. The possibility to drop the maximum in \eqref{Gamma0} for positive values of $k$ can be also explained using the fact that for any $z$ with negative degree, we have $L(z)=0$.

We can show that this infimum can be taken over an even smaller set. We firstly need to introduce a new terminology:

\begin{definition}\label{homog-elt}
	Let $i$ be an integer and $r$ be a real number. Let $\pi$ be an admissible perturbation of the Chern-Simons 
	functional of an integral homology sphere $Y$. Let $\alpha_1$, $\dots$, $\alpha_k$ be the critical points 
	of the perturbed Chern-Simons functional whose Floer gradings are equal to $i$ mod $8$. 
	For each $\alpha_i$, there is a unique path $z_i$ from the trivial connection $\Theta$ to $\alpha_i$ 
	such that $\ind(z)$ is equal to $-i-3$. 
	Let also $r_i$ denote the topological energy of $z_i$, which is an integer lift of $-\CS(\alpha_i)$. (See the identity in 
	\eqref{energy-CS}.) A {\it homogenous} element of $C_*^\pi(Y)$ with {\it weight} $(r,i)$
	is defined to be an element of the following form:
	\[
	  \lambda^r(s_1\lambda^{r_1}\alpha_1+s_2\lambda^{r_2}\alpha_2+\dots+s_k\lambda^{r_k}\alpha_k)
	\] 
	where $s_i$ are rational numbers. The set of all homogenous elements of $C_*^\pi(Y)$ with weight $(r,i)$
	is denoted by $C_{(r,i)}^\pi(Y)$. 
\end{definition}	
\begin{definition}	
	 Given: 
	\[
	  \alpha=\sum_{l=1}^k \sum_{j}s_{l,j}\lambda^{r_{l,j}}\alpha_l \in C_i^\pi(Y)
	\]  
	define:
	\[
	  {\rm P}_{r,i}(\alpha):=\sum_{l=1}^k \sum_{j}s_{l,j}{\rm P}_{r,i}(\lambda^{r_{l,j}}\alpha_l) \in C_i^\pi(Y)
	\]  	
	where ${\rm P}_{r,i}(\lambda^{r_{l,j}}\alpha_l)$ is equal to $\lambda^{r_{l,j}}\alpha_l$ if $r_{l,j}=r+r_l$, and is 
	equal to $0$ otherwise. We extend ${\rm P}_{r,i}$ to a map on $C_*^\pi(Y)$ by requiring ${\rm P}_{r,i}(\beta)=0$
	if $\beta$ is an element with Floer grading $j$ such that $j\nequiv i$ mod $8$.
	The map ${\rm P}_{r,i}$ is called {\it projection to homogenous elements of weight $(r,i)$}.
	Similarly, we define a projection map ${\rm P}_r:\Lambda\to \Lambda$ which maps an element $a$ of 
	$\Lambda$ to $s\lambda^r$ where $s$ is the rational coefficient of $\lambda^r$ in $a$.
\end{definition}

The following lemma is a straightforward consequence of additivity of indices and topological energies:
\begin{lemma}\label{proj-homog}
	Any homogenous element of weight $(r,i)$ of $C_*^\pi(Y)$ is also a homogenous element of 
	$(r+1,i+8)$. We also have:
	\begin{itemize}
		\item[(i)] $d$ maps a homogenous element of weight $(r,i)$ to a 
		homogenous element of weight $(r,i-1)$, and we have 
		$d \circ {\rm P}_{r,i}={\rm P}_{r,i-1}\circ d$;
		\item[(ii)] $U$ maps a homogenous element of weight $(r,i)$ to a 
		homogenous element of weight $(r,i-4)$, and we have 
		$U \circ {\rm P}_{r,i}={\rm P}_{r,i-4}\circ U$;
		\item[(iii)] $D_1$ maps a homogenous element of weight $(r,1)$ to a rational multiple of $\lambda^r$, 
		and we have $D_1\circ {\rm P}_{r,1}={\rm P}_{r} \circ D_1$;
		\item[(iv)] $D_2$ maps $\lambda^r$ to a homogenous element of weight $(r,-4)$, and we have
		$D_2\circ {\rm P}_{r}={\rm P}_{r,-4}\circ D_2$.
	\end{itemize}
\end{lemma}

Suppose $\alpha\in C_*^\pi(Y)$ satisfies \eqref{alpha-pos}, and $\mdeg(D_1U^{k-1}(\alpha))=r_0$. We define $\alpha_0$ to be $\lambda^{-r_0}\cdot{\rm P}_{(r_0,4(k-1)+1)}(\alpha)$, which is a homogenous element of weight $(0,4(k-1)+1)$. By Lemma \ref{proj-homog}, $\alpha_0$ satisfies the identities in \eqref{alpha-pos} and $\mdeg(D_1U^{k-1}(\alpha_0))=0$. Since the projection maps ${\rm P}_{r,i}$ do not decrease $\mdeg$, we can conclude:
\[
  \mdeg(D_1U^{k-1}(\alpha_0))- \mdeg(\alpha_0)\leq \mdeg(D_1U^{k-1}(\alpha))-\mdeg(\alpha)
\]
This inequality implies that to find the value of $\Gamma_Y(k)$, it suffices to take the infimum in \eqref{Gamma1} over the following set:
\begin{equation}\label{Lpik-pos}
	\mathcal L^\pi_k:=\{\alpha\in C_*^\pi(Y) \mid \alpha \in C_{(0,4k-3)}^\pi(Y),\,
	d \alpha=0,\,\, (k-1)=\min \{j\mid D^1U^{j}(\alpha)\neq 0\}\}
\end{equation}

We can give a similar description for $\Gamma_Y(k)$ in the case that $k$ is non-positive. Let:
\begin{equation} \label{pair-2-0}
	\alpha\in C_*^\pi(Y) \hspace{1cm}\fA=\{a_0,a_1,\dots,a_{-k}\}\subset \Lambda
\end{equation}
be given such that:
\begin{equation} \label{cond-K}
  d \alpha=\sum_{i=0}^{-k}U^{i}D_2(a_i).
\end{equation}
Then for any arbitrary sequence $\{b_i\}_{i=-\infty}^{-1}$ the pair: 
\begin{equation} \label{pair-2}
  z=\sum_{i=0}^{-k} a_ix^i + \sum_{i=-\infty}^{-1}U^{-i-1} D_1(\alpha)x^i\in\brC_*^{\pi}(Y) \hspace{1cm} 
  w=(\alpha,\sum_{i=-\infty}^{-1}b_ix^i)\in \crC_*^{\pi}(Y)
\end{equation}
satisfies the identity $i(z)=\widecheck d(w)$, and any such pair with $z$ being of degree $-k$ is given as in \eqref{pair-2}. Then we have:
\begin{equation} \label{Gamma1-neg}
	\Gamma_Y(k):=\max \left(0, \lim_{|\pi|_{\mathcal P}\to 0}\inf_{(\alpha,\{a_0,a_1,\dots,a_{-k}\})}
	(\mdeg(a_0,a_1,\dots,a_{-k})-\mdeg(\alpha))\right)
\end{equation}
where the infimum is taken over all pairs $(\alpha,\fA)$ which satisfy \eqref{cond-K}.

We can work with a smaller set of pairs $(\alpha,\fA)$ in \eqref{Gamma1-neg}. Let $(\alpha,\fA)$ be as in \eqref{pair-2-0} satisfying \eqref{cond-K}. We also assume that $a_{i_0}$ has the minimum $\mdeg$ among the elements of $\fA$, which is equal to $r_0$. We define:
\[
  \alpha_0:=\lambda^{-r_0}{\rm P}_{r_0,-3-4i_0}(\alpha)\hspace{1cm}b_i=\left\{
  \begin{array}{ll}
 	\lambda^{-r_0}{\rm P}_{r_0+\frac{i_0-i}{2}}(a_i)&i\equiv i_0 \mod 2\\
	0&i\nequiv i_0 \mod 2
  \end{array}
  \right.
\]
Our assumption on $i_0$ implies that $b_i=0$ if $i>i_0$. Projection of \eqref{cond-K} to homogenous elements of weight $(r_0,-4-4i_0)$ shows that:
\begin{equation} \label{cond-K-0}
  d \alpha_0=\sum_{i=0}^{-k}U^{i}D_2(b_i).
\end{equation}
We also have $\mdeg(b_0,b_1,\dots,b_{-k})=0$ and $\mdeg(\alpha_0)$ is not less than $\mdeg(\alpha)-r_0$ . Consequently, we have:
\[
  (\mdeg(b_0,b_1,\dots,b_{-k})-\mdeg(\alpha_0))\leq (\mdeg(a_0,a_1,\dots,a_{-k})-\mdeg(\alpha))
\]
This analysis shows that it suffices to take the infimum in \eqref{Gamma1} over the following set:
\begin{align}
	\mathcal L^\pi_k:=\{(\alpha,\{a_0,a_1,\dots,a_{-k}\}) \mid &\alpha \in C_{(0,4k-3)}^\pi(Y),\,
	d \alpha=\sum_{i=0}^{-k}U^{i}D_2(a_i),\text{ $a_i$ is a rational multiple} \nonumber\\
	&\text {of $\lambda^{\frac{-k-i}{2}}$ if $i\equiv k$ mod $2$, and is zero otherwise.} \} \label{Lpik-neg}
\end{align}

\begin{remark}
	The author does not know any example where the value of $\Gamma_Y$ at a non-positive integer is a finite positive
	number. However, one can easily construct chain complexes over $\Lambda$ with the similar formal properties 
	as $\widetilde C_*^\pi(Y)$ such that the analogue of $\Gamma_Y$ takes non-trivial values at non-positive integers.
	For instance, let $(C_*,d)$ be a $\Z/8\Z$-graded complex generated by two generators $\alpha$ and $\beta$ in 
	degrees $5$ and $4$ such that:
	\[
	  d\alpha=\lambda^{r_1}\beta\hspace{1cm} d\beta=0
	\]
	We also define $U:C_*\to C_{*-4}$ and $D_1:C_*\to \Lambda$ to be trivial maps. Let also $D_2:\Lambda \to C_*$
	be defined as follows:
	\[
	  D_2(1)=\lambda^{r_2}\beta
	\]
	where $r_2$ is a real number smaller than $r_1$. As in \eqref{Gamma0}, we can associated a map 
	$\Gamma:\Z\to \overline \R^{\geq 0}$. The value of this function at $0$ is equal to the positive number $r_1-r_2$
	(and this is the only non-trivial value of $\Gamma$). We hope that the understanding the behavior of $\Gamma_Y$ with 
	respect to various topological constructions, such as surgery along knots, would be helpful to study whether there are 
	complexes as above with non-trivial values of $\Gamma_Y$ at non-positive integers.
\end{remark}
Motivated by the definition of $\mathcal L^\pi_k$ in \eqref{Lpik-pos} and \eqref{Lpik-neg}, we also give the following definition:
\begin{definition}
	Let $\pi$ be an admissible perturbation for the Chern-Simons functional of an integral homology sphere
	$Y$. We say $w\in \crC_*^\pi(Y)$ and $z\in \brC_*^\pi(Y)$ form a {\it special pair 
	of degree $N$} if the following conditions are satisfied:
	\begin{itemize}
		\item[(i)] ${\rm Deg}(z)=N$;
		\item[(ii)] $w=(\alpha,0)$ and $z=\sum_{j=-\infty}^{N}a_jx^j$;
		\item[(iii)] if $N<0$, then $\alpha\in \mathcal L^\pi_{-N}$ and if $N\geq 0$, then 
		$(\alpha,\{a_0,\dots,a_N\})\in \mathcal L^\pi_{-N}$.
		\item[(iv)] For $j\leq -1$, we have $a_j=D_1U^{-j-1}(\alpha)$.
	\end{itemize}
\end{definition}
In particular, any special pair of degree $N$ satisfies the identity $i(z)=\widecheck d(w)$. The above discussion shows that to compute the value of $\Gamma_Y$ at $-N$, it suffices to take the infimum in \eqref{Gamma0} over all special pairs $(z,w) $ of degree $N$.

\begin{example} \label{simple-examples}
	In the case that $Y=S^3$, $\Sigma(2,3,5)$ or $-\Sigma(2,3,5)$, we fix the metrics on $Y$ which are induced by the 
	standard metric on the 3-dimensional sphere, the universal cover of $Y$. 
	Then all critical points are non-degenerate and all moduli spaces are regular.
	Therefore, we can use the trivial perturbation $\pi_0$ to compute $\Gamma_Y$. In the case that $Y=S^3$, the set 
	$\mathcal L_k^{\pi_0}$ is empty for any positive $k$ and consists of elements of the form $(0,\{a_0,\dots,a_{-k}\})$ for negative values of
	$k$. This implies that: 
	\[
	  \Gamma_{S^3}(k)=\left\{
	  \begin{array}{ll}
		  \infty&k>0\\
		  0&k\leq 0
	  \end{array}		  
	  \right.
	\]
	
	The complex $C_*^{\pi_0}(\Sigma(2,3,5))$ is generated by two flat connections $\alpha$ and $\beta$ with Floer gradings 
	$1$ and $5$. Then we have:
	\[
	  D_1(\alpha)=\lambda^{\frac{1}{120}} \hspace{1cm}U(\beta)=8\lambda^{\frac{2}{5}}\alpha\hspace{1cm}D_2=0
	\]  
	The above identities after evaluating $\lambda$ at $1$ are verified in \cite{Fro:h-inv}. 
	The calculation of the index and the Chern-Simons
	functional which determines the powers of $\lambda$ in the above identities can be found, for example, in \cite{FS:HFSF}.
	For an exposition of the method of \cite{FS:HFSF}, we refer the reader to \cite[Subsection 3.4]{AX:suture-higher} 
	where the same conventions as here 
	for the definition of the Chern-Simons functional and Floer grading is used.
	These calculations show that 
	The complex $C_*^{\pi_0}(-\Sigma(2,3,5))$ is generated by two flat connections $\alpha^*$ and $\beta^*$ with Floer gradings
	$4$ and $0$. We also have: 
	\[
	  D_2(1)=\lambda^{\frac{1}{120}}\alpha^* \hspace{1cm}U(\alpha^*)=8\lambda^{\frac{2}{5}}\beta^*\hspace{1cm}D_1=0
	\]	
	These identities show that:
	\[
	  \Gamma_{\Sigma(2,3,5)}(k)=\left\{
	  \begin{array}{ll}
		  \infty&k>2\\
		  \frac{49}{120}&k=2\\
		  \frac{1}{120}&k=1\\
		  0&k\leq 0
	  \end{array}		  
	  \right.	  
	\]
	and $\Gamma_{-\Sigma(2,3,5)}=\Gamma_{S^3}$.
\end{example}

\subsection{Properties of $\Gamma_Y$} \label{prop-GammaY}
In this subsection, we review some of the basic properties of $\Gamma_Y$. 
\begin{prop}
	$\Gamma_Y$ is an increasing function with values in $\overline \R^{\geq 0}$.
\end{prop}
\begin{proof}
	It is obvious from the definition that $\Gamma_Y$ takes values in $\overline \R^{\geq 0}$.
	To show monotonicity of $\Gamma_Y$, fix an admissible perturbation $\pi$ of the Chern-Simons 
	functional and let $(z_0,w_0)$ be a special pair of degree $-k$. In particular, $z_0$ has degree $-k$, $w_0=(\alpha,0)$
	for an appropriate $\alpha\in C_*^\pi(Y)$ and $\widecheck d(w_0)=i(z_0)$. We define $z_1:=x\cdot z_0$
	and $w_1:=x\circ w_0$. Then $z_1$ is an element of $\brC_*^\pi(Y)$ with degree $-k+1$.
	We also have $i(z_1)= \widecheck d(w_1)$ because $i$ and $\widecheck d$ are 
	$\Lambda[x]$-module homomorphisms. 
	
	It is clear that $\mdeg(z_1)\leq \mdeg(z_0)$. Moreover, for a given $\delta$, there is $\epsilon$ such that for any 
	$\epsilon$-admissible perturbation:
	\[
	  \mdeg(w_1)\geq \mdeg(w_0)-\delta.
	\]	
	These inequalities imply that:
	\[
	  \mdeg(z_1)- \mdeg(w_1)\leq \mdeg(z_0)- \mdeg(w_0)+\delta
	\]
	Consequently, taking infimum over all choices of of $(z_0,w_0)$ and letting $|\pi|_{\mathcal P}$ 
	go to $0$ allow us to conclude that $\Gamma_{Y}(k-1)\leq \Gamma_{Y}(k)$.
\end{proof}

\begin{prop}\label{possible-values-Gamma}
	Suppose $h$ denotes the Fr\'oyshov's $h$-invariant of an integral homology sphere $Y$. 
	Then we have $\Gamma_Y(k)$ is a finite number if and only if $k\leq 2h(Y)$.
\end{prop}

Before giving the proof of the above proposition, we give an interpretation for $h$ in terms of some of the terminology introduced in this paper:

\begin{lemma}\label{non-empty-special-set}
	For any integral homology $Y$ and any admissible perturbation $\pi$, the integer $2h(Y)$ is equal to
	the largest value of $k$ such that the set $\mathcal L^\pi_k$ is non-empty.
\end{lemma}
\begin{proof}
	For any admissible perturbation $\pi$, let $(\fC_*^\pi(Y),\fd)$ be the standard 
	Floer chain complex defined over rational numbers.
	The definition of this complex is similar to $C_*^\pi(Y)$ with the difference that 
	we use rational numbers as the coefficient ring and we
	do not include any power of $\lambda$ in the definition of 
	the differential $\fd$. By dropping the powers of $\lambda$ from the definition 
	of $D_1$, $D_2$ and $U$, we can also define:
	\[
	  \fD_1:\fC_1^\pi(Y) \to \Q\hspace{1cm}\fD_2:\Q \to \fC_4^\pi(Y) \hspace{1cm}\fU:\fC_*^\pi(Y) \to \fC_{*-4}^\pi(Y)
	\]
	Using these operators, we define:
	\[
	  \mathcal K^\pi_k:=\{\alpha\mid\alpha\in \fC_{4k-3}^\pi,\,\,
	\fd \alpha=0,\,\, (k-1)=\min_{\fD^1\fU^{j}(\alpha)\neq 0}(j)\}
	\]
	for positive values of $k$, and: 
	\begin{align*}
		\mathcal K^\pi_k:=\{(\alpha,\{s_0,s_1,\dots,s_{-k}\}) \mid &\alpha \in \fC_{4k-3}^\pi(Y),\,\,, s_i\in \Q,\,\,
	\fd \alpha=\sum_{i=0}^{-k}\fU^{i}\fD_2(s_i), \nonumber\\
	&\text{ $s_i\neq 0$ only if $i\equiv k$ mod $2$.} \}
	\end{align*}	
	Evaluation of $\lambda$ at $1$ gives a bijective map from $ \mathcal L^\pi_k$ to $ \mathcal K^\pi_k$ for any integer $k$. 
	It is also shown in \cite[Proposition 4]{Fro:h-inv} that the set $ \mathcal K^\pi_k$ is non-empty only if 
	$k\leq 2h(Y)$.\footnote{In \cite{Fro:h-inv}, 
	the cohomological convention is used in the definition of Floer chain complexes. The reader should take that into account 
	for comparing \cite{Fro:h-inv} and the present article.} Thus 
	the largest $k$ with $\mathcal L^\pi_k$ being non-empty is equal to $2h(Y)$.
\end{proof}

\begin{proof}[Proof of Proposition \ref{possible-values-Gamma}]
	Lemma \ref{non-empty-special-set} shows that if $k>2h(Y)$, then $\mathcal L^\pi_k$ is empty for any admissible perturbation 
	$\pi$. This lemma also implies that if $k\leq 2h(Y)$, then $\mathcal L^\pi_k$ is non-empty for any given $\epsilon$-admissible 
	perturbation $\pi$. Therefore, the infimum in \eqref{Gamma1} or \eqref{Gamma1-neg} is finite for any given $\pi$.
	Now we can follow the argument in the proof of 
	Proposition \ref{hom-inv-Gamma} and obtain a uniform upper bound for the infimum of \eqref{Gamma1} or \eqref{Gamma1-neg} 
	for any other $\epsilon$-admissible perturbation. Thus $\Gamma_Y(k)$ is a finite number.
\end{proof}

\begin{remark}\label{ins-h-mon-h}
	Combining Lemma \ref{non-empty-special-set} and the discussion of the previous section shows that 
	$h(Y)$ is given by the following identity:
	\begin{align*}
	  h(Y)&=\max \{k\mid\exists z\in \brI_*(Y) \text{ with ${\rm Deg}(z)=-k$ and } i_*(z)\neq 0\}\\
	  &=\min \{k\mid\forall z\in \brI_*(Y) \text{ with ${\rm Deg}(z)=-k$ and $z\in {\rm image}(j_*)$}\}-1
	\end{align*}
	This definition of $h$ is similar to the standard definition of monopole $h$-invariant 
	\cite[Section 39]{KM:monopoles-3-man}
	 and the correction term in Heegaard Floer homology \cite{OzSz:d-inv}.
\end{remark}

\begin{prop}\label{CS-val}
	For any integral homology sphere $Y$ and any integer $k$, either $\Gamma_Y(k)=\infty$, $\Gamma_Y(k)=0$, or there is an irreducible 
	flat connection $\alpha$ such that $\Gamma_Y(k)$ is equal to $\CS(\alpha)$ mod $\Z$.
\end{prop}
\begin{proof}
	Suppose $\pi$ is an admissible perturbation of the Chern-Simons functional of $Y$. Firstly let $k$ be a positive number and 
	$\alpha\in \mathcal L_k^\pi$. The condition $\alpha\in C_{(0,4k-3)}^\pi(Y)$ implies that:
	\[
	  \alpha=s_1\lambda^{r_1}\alpha_1+s_2\lambda^{r_2}\alpha_2+\dots+s_j\lambda^{r_j}\alpha_j
	\]
	where $\alpha_1$, $\dots$, $\alpha_k$ are the critical points of $\CS+f_\pi$ with Floer grading $4k-3$, $s_1$, $\dots$, $s_k$ 
	are rational numbers. The exponent $r_i$ is introduced in Definition \ref{homog-elt} and is equal to $-\CS(\alpha_i)$ mod $\Z$.
	Since $\alpha\in C_{(0,4k-3)}^\pi(Y)$, we have $\mdeg(D_1U^{k-1}(\alpha))=0$. We also have:
	\[
	  \mdeg(\alpha)=\min_{s_i\neq 0}(r_i).
	\]
	In particular, we have:
	\[
	  \inf_{\alpha\in\mathcal L_k^\pi}\left(\mdeg(D_1U^{k-1}(\alpha))-\mdeg(\alpha)\right)\in \{-r_1,\dots,-r_k\}.
	\]
	Therefore, there is an irreducible critical point $\alpha_i$ of $\CS+f_\pi$ with index $4k-3$ such that the above infimum is 
	equal to $\CS(\alpha_i)$ mod $\Z$. As we let $\pi\to 0$, the values of the Chern-Simons functional on the critical points 
	of $\CS+f_\pi$ converge to the finite set of values of the Chern-Simons functional on irreducible $\SU(2)$ flat connections 
	\cite[Lemma 3.8]{KM:yaft}.
\end{proof}

\begin{remark}
	It is natural to ask whether there is an integral homology sphere $Y$ such that $\Gamma_Y$ takes irrational values. 
	Proposition \ref{CS-val} implies that if $\Gamma_Y$ takes an irrational value, then the value of the Chern-Simons functional 
	of an $\SU(2)$-flat connection is irrational. Currently, it is unknown whether there is such a flat connection on a 3-manifold $Y$.
	For example, the values of the Chern-Simons functional for any plumbed 3-manifold
	takes rational values on $\SU(2)$-flat connections. 
	Consequently, if $\Gamma_Y$ is not rational valued, then $Y$ has to be linearly independent of plumbed 3-manifolds.
\end{remark}

\begin{definition}\label{tau}
	For any integral homology sphere $Y$, $\tau(Y)$ is defined to be:
	\begin{equation}\label{z-tau}
		\tau(Y):=\inf_{z}\{\mathcal E(z)\}
	\end{equation}
	where the infimum is taken over all paths $z$ from an irreducible flat connection $\alpha$ to the trivial 
	connection $\Theta$ such that the moduli space $\mathcal M_z(\alpha,\Theta)$ is non-empty.
\end{definition}	

Given a path $z$ from $\alpha$ to $\Theta$ as in \eqref{z-tau}, $\mathcal E(z)$ is equal to $\CS(\alpha)$ mod $\Z$. Thus $\tau(Y)$ takes values in a discrete set because the Chern-Simons functional takes only finitely many values on the set of flat $\SU(2)$-connections. This implies that $\tau(Y)$ can be realized by the topological energy of an ASD connection from an irreducible flat connection $\alpha$ to the trivial connection. In particular, this constant is positive and we have:
\begin{equation} \label{tau-gr}
 	\tau(Y)\geq \min\{r\mid r\in \R^{>0},\,r\equiv\CS(\alpha)\,\,\text{ mod $\Z$\,\, for a flat connection }\alpha\}
\end{equation}
We can use $\tau(Y)$ to give a constraint for the values of $\Gamma_Y$ at positive integers:
\begin{prop} \label{lower-bound-tau}
	For any integral homology sphere $Y$, we have $\Gamma_Y(1)\geq \tau(Y)$. Consequently, for any positive 
	integer $k$, we have $\Gamma_Y(k)\geq \tau(Y)$.
\end{prop}

The proof of the above proposition needs some preparation. Firstly we start with a standard exponential decay result about instantons on tubes:
\begin{lemma}\label{exp-decay}
	Let $\alpha$ be a non-degenerate $\SU(2)$-flat connection on an integral homology sphere $Y$.
	There are constants $\epsilon_0$, $\epsilon_1$, $C_l$ and $\delta$ such that the following holds:
	\begin{itemize}
		\item[(i)] Suppose $A$ is an ASD connection on $(0,\infty)\times Y$ such that the $L^2$-distance
		between $A|_{\{t\}\times Y}$, for $t\in (0,1)$, and $\alpha$ is less than $\epsilon_0$ and 
		$|\!|F(A)|\!|_{L^2((0,\infty)\times Y)}<\epsilon_1$.
		Then $A$ is gauge equivalent to a connection of the form $\alpha+a$ where $a$ is a 1-form on 
		$(0,\infty)\times Y$ with values in $\su(2)$ and:
		\begin{equation}\label{exp-decay-ineq-1}
		  |\nabla^la|(t,y)\leq C_le^{-\delta\cdot t }|\!|F(A)|\!|_{L^2((0,1)\times Y)}.
		\end{equation}
		for $t\in (\frac{1}{2},\infty)$.
		\item[(ii)] Suppose $A$ is an ASD connection on $(-T,T)\times Y$, for $T>1$, such that the $L^2$-distance
		between $A|_{\{t\}\times Y}$, for $t\in (-T,-T+1)$, and $\alpha_0$ is less than $\epsilon_0$ and 
		$|\!|F(A)|\!|_{L^2((-T,T)\times Y)}<\epsilon_1$.
		Then $A$ is gauge equivalent to a connection of the form $\alpha+a$ where $a$ is a 1-form on 
		$(-T,T)\times Y$ with values in $\su(2)$ and:
		\begin{equation}\label{exp-decay-ineq-2}
		  |\nabla^la|(t,y)\leq C_le^{-\delta (T-|t|) }\(|\!|F(A)|\!|_{L^2((-T,-T+1)\times Y)}+|\!|F(A)|\!|_{L^2((T-1,T)\times Y)}\).
		\end{equation}
		for $t\in (-T+\frac{1}{2},T-\frac{1}{2})$.
	\end{itemize}
\end{lemma}
\begin{proof}
	This lemma is essentially proved in \cite[Theorem 4.2, Proposition 4.3 and Proposition 4.4]{Don:YM-Floer}. 
	In \cite{Don:YM-Floer}, it is assumed that $Y$ is $\SU(2)$-non-degenerate. We weaken 
	this assumption by requiring that only the given flat connection $\alpha$ is non-degenerate and make the additional 
	assumption that the $L^2$ distance of $A|_{\{t\}\times Y}$ and $\alpha$ for appropriate values of $t$ is bounded by 
	$\epsilon_0$. Since $\alpha$ is non-degenerate, the $L^2$-distance of $\alpha$ and any other flat connection is 
	greater than a positive number $\kappa$. Let $\epsilon_0$ be equal to $\frac{\kappa}{2}$. 
	We claim that for any positive constant $\eta$ smaller than $\epsilon_0$ and any positive integer $k$, 
	there is a constant $\epsilon_1$ such that for any ASD 
	connection $A$ as in part (i) (resp. part (ii)) of the lemma, the $L^2_k$-distance 
	between the connections 
	$A|_{\{t\}\times Y}$ and $\alpha$ is less than $\eta$ for any $t\in (\frac{1}{2},\infty)$ 
	(resp. $t\in (-T+\frac{1}{2},T-\frac{1}{2})$).
	
	We prove this claim for part (i). The proof for the other case is similar. Suppose there is a sequence of ASD connections $A_i$
	on $(0,\infty)\times Y$ such that $|\!|F(A_i)|\!|_{L^2((0,\infty)\times Y)}\to 0$, the $L^2$-distance of $A_i|_{\{t\}\times Y}$
	for $t\in (0,1)$ is less than $\epsilon_0$, and there is $t_i\in [1,\infty)$ such that the $L^2_k$-distance between 
	$A_i|_{\{t\}\times Y}$ and the space of flat connections is equal to $\eta$. 
	By Uhlenbeck compactness theorem, the connections 
	$A_i|_{(t_i-1,t_i-1)\times Y}$, after passing to a subsequence and changing gauge, are $C^\infty_{\loc}$-convergent to a 
	flat connection on $(-1,1)\times Y$. However, it is a contradiction because the $L^2_k$ distance between 
	$A_i|_{\{t\}\times Y}$ and any flat connection is at least $\epsilon_0$. This verifies the claim.
	Given this claim, the arguments of \cite{Don:YM-Floer} can be applied to prove the lemma without any change.
\end{proof}

\begin{lemma}\label{tau-pos}
	Suppose $\{\pi_i\}_i$ is a sequence of admissible perturbations of the Chern-Simons functional of $Y$ such that 
	$|\pi_i|_{\mathcal P}\to 0$. Suppose $A_i\in \mathcal M_{z_i}^{\pi_i}(\alpha_i,\Theta)$ where $\alpha_i$ 
	is an irreducible critical point of $\CS+f_{\pi_i}$ and $z_i$ is a path with index $1$. 
	We also assume that there is a constant $\nu$ such that $f_{\pi_i}$ 
	vanishes for connections whose $L^2$-distance to $\Theta$ is less than $\nu$.
	Then there exists 
	an irreducible flat connection $\alpha_0$, a path $z_0$ from $\alpha_0$ to $\Theta$, and 
	a (non-perturbed) ASD connection $A_0\in\mathcal M_{z_0}(\alpha_0,\Theta)$ such that:
	\begin{equation}\label{energy-bd}
	  \mathcal E(A_0)\leq \limsup_{i} \mathcal E(A_i)
	\end{equation}
\end{lemma}

\begin{proof}
	The connection $A$ can be constructed as a limit of a sequence of connections associated to the connections 
	$\{A_i\}$. The argument is an adaptation of \cite[Theorem 5.4]{Don:YM-Floer}. We divide the proof into several steps:
	
	{\bf Step 1:} {\it The sequence $|\!|F(A_i)|\!|_{L^2(\R\times Y)}$ is bounded.} 
	
	It suffices to show that the topological energies of paths between critical points of $\CS+f_{\pi_i}$ with index less than a 
	fixed integer $N$ are uniformly bounded by a constant $K$ which does not depend on $i$. In the case that we are 
	concerned only with the trivial perturbation, this is standard.
	The case of non-trivial perturbations $\{\pi_i\}$ can be reduced to the case of the trivial perturbation using the following trick.
	
	Each $\SU(2)$-flat connection $\alpha$ has a path-connected open neighborhood such that the index and the topological 
	energy of paths in this neighborhood are uniformly bounded. We cover the space of flat connections by finitely many 
	such open neighborhoods. If $|\pi_i|_{\mathcal P}$ is small enough, then each end point of the path $z_i$ belong to
	one of the chosen open sets. Now we can use additivity of indices and topological energies 
	with respect to concatenation of paths to verify the claim.	
	
	{\bf Step 2:} {\it There is a positive constant $\epsilon$ and for any connection $A_i$ in the above sequence, 
	there is a constant $T_i$
	such that $A_i|_{(T_i,\infty)\times Y}$ is an ASD connection and:
	\begin{equation}\label{eq}
	  \int_{(T_i,\infty)\times Y}|F(A_i)|^2\dvol=\epsilon
	\end{equation}
	} 
	
	Suppose $\epsilon_0$ and $\epsilon_1$ are given as in Lemma \ref{exp-decay}. As in the proof of Lemma \ref{exp-decay},
	we assume that $\epsilon_0$ is smaller than half of the $L^2$-distance of the trivial connection and the space of irreducible
	flat $\SU(2)$-connections of $Y$. Let $\epsilon':=\min(\epsilon_0,\epsilon_1,\frac{\nu}{2})$ and $T_i'$ be 
	the largest real number 
	such that the $L^2$ distance between $A|_{\{T_i'\}\times Y}$ and 
	$\Theta$ is at least $\epsilon'$. Then the connection $A_i|_{(T_i',\infty)\times Y}$ is an ASD connection. 
	We claim that there is a constant $N_0$, independent of $i$, such that:
	\begin{equation}\label{N0}
	  \int_{(T_i',\infty)\times Y}|F(A_i)|^2\dvol\geq \frac{\epsilon'}{N_0}.
	\end{equation}
	Given $i$, if the above inequality does not hold for $N_0\geq 1$, then part (i) of Lemma \ref{exp-decay} 
	implies that there is $a_i$ such that $A_i$ is gauge equivalent to $\alpha+a_i$ and:
	\[
	  |\nabla^la_i|(t,y)\leq C_le^{-\delta\cdot t }|\!|F(A_i)|\!|_{L^2((0,1)\times Y)}.
	\] 
	In particular, we have:
	\[
	  |a_i|(T_i',y)\leq C_0 \frac{\epsilon'}{N_0}
	\] 	
	If $N_0$ is large enough in compare to $C_0$, then the $L^2$ norm of $a_i$ is less than $\frac{\kappa}{2}$,
	which is a contradiction. Thus the inequality holds for an appropriate value of $N_0$ and we define 
	$\epsilon:=\frac{\epsilon'}{N_0}$. There is also a unique value of $T_i$ greater than $T_i'$ such that \eqref{eq} holds.
	
	{\bf Step 3:} {\it There is a connection $A_0\in\mathcal M_{z_0}(\alpha_0,\Theta)$ satisfying \eqref{energy-bd}.} 
	
	Let $A_i'$ be given by translating $A_i$ in the $\R$ direction by the parameter $T_i$. Then we have:
	\begin{equation}\label{eq-translate}
	  \int_{(0,\infty)\times Y}|F(A_i')|^2\dvol=\epsilon
	\end{equation}
	Using Lemma \ref{exp-decay}, the connection $A_i'$ is gauge equivalent to a connection of the form $\alpha+a_i'$
	where $a_i'$ satisfies the inequalities in \eqref{exp-decay-ineq-1}. Thus the Arzela-Ascoli theorem implies that
	there is a subsequence\footnote{Here and in what follows, we always denote a subsequence with the same notation
	as the original sequence.} of the connections which is $C^\infty$-convergent on $(0,\infty)\times Y$. 
	We can also employ the Uhlenbeck compactness theorem 
	to show that there is an ASD connection $A_0'$ and a finite subset $S$ of $\R\times Y$
	such that the sequence $A_i'$, after passing to a subsequence and changing gauge, is $L^p_1$-convergent\footnote{
	The weaker $L^p_1$-convergence instead of $C^{\infty}$-convergence is due to non-locality of holonomy perturbations.
	For a discussion related to adapting the Uhlenbeck compactness theorem to the ASD equation perturbed by 
	holonomy perturbations, see \cite{K:higher}.} 
	on compact subsets of 
	$\R\times Y\backslash S$ to an ASD connection $A_0'$.
	Moreover, we have:
	\[
	  \mathcal E(A_0')\leq \limsup_{i}\mathcal E(A_i)
	\]
	
	These two observations show that $A_i'$ is strongly convergent (resp. weakly convergent) on $(0,\infty)\times Y$ 
	(resp. $\R\times Y$) to $A_0'$. In particular, $A_0'$ is asymptotic to $\Theta$ on the outgoing end of $\R\times Y$.
	Since $A_0'$ is an ASD connection, it is convergent to a flat connection 
	on the incoming end \cite[Theorem 4.18]{Don:YM-Floer}. If this 
	flat connection is irreducible, then we are done. In the case that $A_0'\in \mathcal M_{z_0'}(\Theta,\Theta)$, 
	the term $|\!|F(A_0')|\!|_{L^2((-\infty,-h)\times Y)}$ is strictly less than $\liminf_i|\!|F(A_i')|\!|_{L^2((-\infty,-h)\times Y)}$
	for any value of $h$. Otherwise, the incoming flat connection of $A_i'$ 
	is convergent to the trivial connection which is 
	a contradiction. The connection $A_0'$ is non-trivial and hence its energy is at least $8\pi^2$, 
	the energy of a single instanton.
	
	Let $h$ be chosen large enough such that $S$ is disjoint from $(-\infty,-h)\times Y$ and the distance between 
	$A_i|_{\{t\}\times Y}$ and $\Theta$ is less than $\epsilon_0$ for $t\in (-h-1,-h)$. Define:
	\[
	  \eta =\min(\epsilon_1,
	  \frac{\liminf_i|\!|F(A_i')|\!|_{L^2((-\infty,-h)\times Y)}-|\!|F(A_0')|\!|_{L^2((-\infty,-h)\times Y)}}{2})
	\]
	where $\epsilon_1$ is given by Lemma \ref{exp-decay}. For large values of $i$, we may pick a constant $S_i$ such that:
	\[
	  \int_{(S_i,-h)\times Y} |F(A_i')|^2=\eta+\int_{(-\infty,-h)\times Y} |F(A_0')|^2
	\]
	The constants $S_i$ are convergent to $-\infty$ and we define a new sequence of connections $A_i''$ by translating $A_i'$
	in the $\R$ direction by the parameter $S_i$. Another application of Uhlenbeck compactness implies that $A_i''$,
	after passing to a subsequence and changing gauge, is convergent to an ASD equation which satisfies:
	\[
	  \mathcal E(A_0')+\mathcal E(A_0'')\leq \limsup_{i}\mathcal E(A_i)
	\]
	Moreover, part (ii) of Lemma \ref{exp-decay} implies that $A_0''$ is asymptotic to the trivial connection on the outgoing end.
	If $A_0''$ is asymptotic to an irreducible flat connection on the incoming end, then we are done. Otherwise, we repeat the  
	above process. This process terminates because the energy of a non-trivial ASD connection asymptotic to 
	the trivial connection on both ends is at least $8\pi^2$.
\end{proof}

\begin{proof}[Proof of Proposition \ref{lower-bound-tau}]
	Let $\{\pi_i\}$ be a sequence of admissible perturbations such that $|\pi_i|_{\mathcal P}\to 0$.
	Since the trivial connection is a non-degenerate critical point of the (non-perturbed) Chern-Simons functional of
	$Y$, we may also assume that $f_{\pi_i}$ is trivial for connections whose $L^2$-distance to the trivial connection is less than a 
	fixed constant $\kappa$.
	If $\alpha_i\in \mathcal L_1^{\pi_i}(Y)$, then the difference $\mdeg(D_1(\alpha_i))-\mdeg(\alpha_i)$ 
	is equal to the topological energy of an element $A_i$ of 
	$\mathcal M_{z_i}^{\pi_i}(\alpha_i',\Theta)$ for an irreducible flat connection $\alpha_i'$. 
	Therefore, we can pick the connections $A_i$ such that $\mathcal E(A_i)\to \Gamma_Y(1)$. Applying
	Lemma \ref{tau-pos} to this sequence of connections implies that there is an irreducible flat connection $\alpha_0$
	and an ASD connection $A_0\in\mathcal M_{z_0}(\alpha_0,\Theta)$ such that $\Gamma_Y(1)\geq \mathcal E(A_0)$
	which verifies our claim.
\end{proof}

\begin{definition}\label{tau'}
	Suppose $Y$ is an $\SU(2)$-non-degenerate integral homology sphere.
	We define:
	\begin{equation}\label{z-tau'}
		\tau'(Y):=\inf_{z'}\{\mathcal E(z')\}
	\end{equation}
	where the infimum is taken over all paths $z'$ such that $z'$ is obtained by concatenation of paths 
	$z_1$, $\dots$, $z_k$ satisfying the following properties. There are flat connections 
	$\alpha_0$, $\dots$, $\alpha_k$ such that $\alpha_0$ and $\alpha_k$ are irreducible, $z_i$ is a path from $\alpha_{i-1}$ to 
	$\alpha_i$, the sum of the indices of paths $z_i$ is equal to $4-8n$ for a non-negative integer $n$ 
	and there is an ASD connection 
	$A_i\in\mathcal M_{z_i}(\alpha_{i-1},\alpha_i)$.
\end{definition}

Given a path $z'$ from $\alpha_-$ to $\alpha_+$ as in \eqref{z-tau'}, $\mathcal E(z')$ is equal to $\CS(\alpha_0)-\CS(\alpha_k)$ mod $\Z$. Therefore, $\tau'(Y)$ takes values in a discrete set. This implies that $\tau'(Y)$ can be realized by a tuple $(A_1,\dots,A_k)$ of ASD connections as in the above definition. We also have:
\begin{equation*} \label{tau'-gr}
 	\tau'(Y)\geq \min\{r\mid r\in \R^{>0},\,r\equiv\CS(\alpha_-)-\CS(\alpha_+)\,\,\text{ mod $\Z$\,\, for irreducible flat connections }\alpha_{\pm}\}
\end{equation*}

\begin{lemma}\label{tau'-pos}
	Suppose $Y$ is an $\SU(2)$-non-degenerate integral homology sphere.
	Suppose $\{\pi_i\}_i$ is a sequence of perturbations of the Chern-Simons functional of $Y$ such that 
	$|\pi_i|_{\mathcal P}\to 0$, the critical points of $\CS+f_{\pi_i}$
	are the same as the critical points of the unperturbed Chern-Simons functional and $f_{\pi_i}=0$ 
	in a neighborhood\footnote{This neighborhood is defined using the $L^2$-distance.} of flat connections. 
	Suppose $A_i\in \mathcal M_{z_i}^{\pi_i}(\alpha_{i,-},\alpha_{i,+})$ 
	where $\alpha_{i,-}$ and $\alpha_{i,+}$ are irreducible flat connections and 
	$z_i$ is a path from $\alpha_{i,-}$ to $\alpha_{i,+}$ with index $4$. 
	Then there are distinct flat connections $\alpha_{0}$, $\dots$, $\alpha_{k}$, 
	a path $z_{0,i}$ from $\alpha_{i-1}$ to $\alpha_{i}$ , and 
	a (non-perturbed) ASD connection $A_i\in\mathcal M_{z_{0,i}}(\alpha_{i-1},\alpha_i)$ such that $\alpha_0$ and $\alpha_k$
	are irreducible, the sum of the indices 
	of $z_{0,1}$, $\dots$, $z_{0,k}$ is equal to $4-8n$ for a non-negative integer $n$ and:
	\begin{equation}\label{energy-bound-tau'}
	  \mathcal E(z_{0,1})+\dots +\mathcal E(z_{0,k})\leq \limsup_{i} \mathcal E(z_i).
	\end{equation}
\end{lemma}
\begin{proof}
	The proof is similar to that of Proposition \ref{tau-pos}. Since $Y$ is $\SU(2)$-non-degenerate, there are only finitely many 
	$\SU(2)$-flat connections. Using this observation and the fact that the index of connections $A_i$ are at most $4$,
	we may assume that $A_i\in \mathcal M_{z}^{\pi_i}(\alpha_-,\alpha_+)$
	for a fixed path $z$ and irreducible flat connections $\alpha_-$, $\alpha_+$. Following the argument of 
	Proposition \ref{tau-pos}, we can construct nontrivial ASD connections $A_i\in \mathcal M_{z_{0,i}}(\alpha_{i-1},\alpha_i)$
	where $\alpha_i$ is a flat connection, $\alpha_0=\alpha_-$, $\alpha_k=\alpha_+$, $z_{0,i}$ is a path from 
	$\alpha_{i-1}$ to $\alpha_i$ and the inequality \eqref{energy-bound-tau'} is satisfied.	
\end{proof}

\begin{prop}\label{tau'-str}
	If $Y$ is an $\SU(2)$-non-degenerate integral homology sphere, 
	then for any positive integer $i$ we have $\Gamma_Y(i+1)\geq \Gamma_Y(i)+\tau'(Y)$.
\end{prop}
\begin{proof}
	We fix a sequence of admissible perturbations $\pi_i$ such that $|\pi_i|_{\mathcal P}\to 0$,
	the critical points of $\CS+f_{\pi_i}$ are the same as the critical points of the unperturbed 
	Chern-Simons functional and $f_{\pi_i}=0$ in a neighborhood of flat connections. 
	For each admissible perturbation $\pi_i$ and any $\alpha\in \mathcal L_{k+1}^{\pi_i}(Y)$, we have 
	$U(\alpha)\in \mathcal L_{k}^{\pi_i}(Y)$. Furthermore, there are flat connections $\alpha_{i,-}$,
	$\alpha_{i,+}$ with Floer gradings $4k-3$, $4(k-1)-3$ and $A_i \in \mathcal M^{\pi_i}_{z_i}(\alpha_{i,-},\alpha_{i,+})$
	with index $4$ such that:
	\[
	  \mdeg(U(\alpha))-\mdeg(\alpha)\geq \mathcal E(A_i).
	\]
	By taking infimum over all $\alpha\in \mathcal L_{k+1}^{\pi_i}(Y)$, we conclude that :
	\[
	  \inf_{\alpha\in \mathcal L_{k+1}^{\pi_i}(Y)} (\mdeg(D_1U^{k-1}(\alpha)-\mdeg(\alpha)\geq 
	  \inf_{\alpha'\in \mathcal L_{k}^{\pi_i}(Y)} (\mdeg(D_1U^{k-2}(\alpha')-\mdeg(\alpha')+\inf_{A_i}\mathcal E(A_i).
	\]	
	where the infimum is taken over all index $4$ connections $A_i \in \mathcal M^{\pi_i}_{z_i}(\alpha_{i,-},\alpha_{i,+})$
	where $\alpha_{i,-}$, $\alpha_{i,+}$ are irreducible flat connections. By taking the limit of the above inequality as 
	$i \to \infty$ and using Lemma \ref{tau'-pos}, we can conclude that $\Gamma_Y(i+1)\geq \Gamma_Y(i)+\tau'(Y)$.
\end{proof}

Next, we  use the argument in the proof of Theorem \ref{hom-inv-Gamma} to obtain a more general result. The constant $\eta$ in the following theorem is introduced in Definition \ref{eta}:

\begin{theorem} \label{inc-energy}
	Suppose $W:Y \to Y'$ is a cobordism of integral homology spheres with { $b_1(W)=b^+(W)=0$}. 
	Then $\Gamma_{Y'}(k) \leq \Gamma_{Y}(k)-\eta(W)$ for any positive integer $k$. For a non-positive
	$k$, we have the weaker inequality $\Gamma_{Y'}(k) \leq \max(\Gamma_{Y}(k)-\eta(W),0)$.
\end{theorem}

\begin{definition}\label{eta}
	Let $W$ be a cobordism from an integral homology sphere $Y$ to another integral homology sphere
	$Y'$. Let:
	\begin{equation}
		\eta(W):=\inf_{A}\{\mathcal E(A)\}
	\end{equation}
	where the infimum is taken over all ASD connections $A$ on $W$ which is asymptotic to irreducible flat 
	connections on $Y$ and $Y'$. Here we allow $A$ to be a broken ASD connection. 
	That is to say, $A$ might have one of the following forms for irreducible flat connections $\alpha$ and $\alpha'$:
	\begin{itemize}
		\item[(i)]  $A\in \mathcal M_{z}(W,\alpha,\alpha')$;
		\item[(ii)] $A=(A_0,A_1)$ where $A_0\in \mathcal M_{z_0}(\R\times Y,\alpha,\Theta)$ and 
			$A_1\in \mathcal M_{z_1}(W,\Theta,\alpha')$;
		\item[(iii)] $A=(A_0,A_1)$ where $A_0\in \mathcal M_{z_0}(W,\alpha,\Theta)$ and 
			$A_1\in \mathcal M_{z_1}(\R\times Y',\Theta,\alpha')$;	
		\item[(iv)] $A=(A_0,A_1,A_2)$ where $A_0\in \mathcal M_{z_0}(\R\times Y,\alpha,\Theta)$,
				$A_1\in \mathcal M_{z_1}(W,\Theta,\Theta)$ and 
				$A_2\in \mathcal M_{z_2}(\R\times Y',\Theta,\alpha')$.		
	\end{itemize}
\end{definition}	
The constant $\eta(W)$ is a non-negative number. The set of possible values for the energy of ASD connections on $W$ is a discrete subset of non-negative integers. Therefore, if $\eta(W)$ is finite, then it can be realized by the energy of a (possibly broken) ASD connection on $W$ which is asymptotic to non-trivial connections on both ends. We have:
\begin{align*} \label{eta-gr}
 	\eta(W)\geq \min\{r\mid r\in \R^{\geq0},\,r\equiv\CS(\alpha)-\CS(\alpha')\,\,&\text{ mod $\Z$\,\,}\\ \text{for irreducible flat connections }
	&\alpha,\, \alpha' 
	\text{ on } Y,\,Y'\}
\end{align*}
We also make the observation that if $\eta(W)=0$, then there is a flat connection extending non-trivial flat connections on $Y$ and $Y'$. Consequently, if $W$ is simply connected, then $\eta(W)>0$.

\begin{lemma} \label{limit-small-energy}
	Let $\{\pi_i\}$, $\{\pi_i'\}$ be sequences of $\epsilon_i$-admissible perturbations for the 
	Chern-Simons functionals of $Y$, $Y'$ such that $f_{\pi_i}$ and $f_{\pi_i'}$ vanish in fixed neighborhoods of the 
	trivial connections on $Y$ and $Y'$. Let $\pi_i$ and $\pi_i'$ be extended to an $\epsilon_i$-admissible
	perturbation $\overline \pi_i$ on the cobordism $W:Y\to Y'$, which satisfies { $b_1(W)=b^+(W)=0$}. 
	Let $A_i\in \mathcal M_{z_i}^{\overline \pi_i}(W,\alpha_i,\alpha_i')$ be chosen such that 
	$\alpha_i$, $\alpha_i'$ are respectively 
	irreducible critical points of $\CS+f_{\pi_i}$, $\CS+f_{\pi_i'}$, and $z_i$ is a path along $W$ with index $0$.
	If $\epsilon_i\to 0$, then there is a (possibly broken) ASD connection
	$A_0$ on $W$ which is asymptotic to irreducible flat connections $\alpha_0$, $\alpha_0'$ on $Y$, $Y'$ and: 
	\begin{equation}\label{energy}
	  \mathcal E(A_0)\leq \limsup_{i}\mathcal E(A_i).
	\end{equation}
\end{lemma} 
\begin{proof}
	Firstly as in Step $1$ of the proof of Lemma \ref{tau-pos} we can show that the terms $|\!|F(A_i)|\!|_{L^2(W^+)}$ 
	are uniformly 
	bounded. Thus the Uhlenbeck compactness theorem implies that the connections $A_i$, after passing to a subsequence 
	and changing gauge, are weakly convergent to an ideal instanton $(A_0,S)$ on $W^+$, where 
	$S\subset W^+$ is a finite subset and $A_0$ 
	is an ASD connection. The weakly convergence of $A_i$ implies that $A_i$, after possibly changing the gauge,
	is $L^p_1$ convergent on compact subsets of $W^+\backslash S$ to $A_0$, and:
	\begin{equation*}
		\mathcal E(A_0)\leq \limsup_{i}\mathcal E(A_i)
	\end{equation*}
	If the connection $A_0$ is asymptotic to non-trivial flat connections on both ends, then there is nothing left to prove.
	If it is asymptotic to the trivial connection on one of the ends, say the outgoing end,  
	then we can argue as in Step 3 of Lemma \ref{tau-pos} to find an instanton $A_1$ in a moduli space of the form
	$\mathcal M_{z_1}(\R\times Y',\Theta,\alpha')$ where $\alpha'$ is an irreducible flat connection on $Y'$. The pair 
	$(A_0,A_1)$ also satisfy the analogue of the inequality in \eqref{energy}. The other cases can be treated similarly.
\end{proof}

\begin{proof}[Proof of Theorem \ref{inc-energy}]
	We follow a similar argument as in Theorem \ref{hom-inv-Gamma}. Let $w=(\alpha,0)$ and $z$ form a special pair 
	of degree $-k$. Then the pair $w'=(\varphi(\alpha)+L(z),0)$ and $z'=\brC_W(z)$ satisfies the identity 
	$\widecheck d(w')=i(z')$. Here $\varphi:C_*^{\pi}(Y) \to C_*^{\pi'}(Y')$ is the cobordism map associated to $W$.
	Recall that $L(z)=0$ if $k$ is a positive integer.
	The first inequality in \eqref{1st-ineq} can be modified as follows:
	\[  
	  \mdeg(w')\geq \mdeg(w)+\mathcal E(A)
	\]
	for a connection $A\in \mathcal M_z^{\overline \pi}(\alpha,\alpha')$ of index $0$ 
	where $\alpha$, $\alpha'$ are irreducible critical points of
	$\CS+f_{\pi_i}$, $\CS+f_{\pi_i'}$ and $z$ is a path over $W$ from $\alpha$ to $\alpha'$ with index $0$. 
	This the inequality in
	\eqref{ineq-cob} can be improved as follows:
	\begin{align*}
	  \inf_{\substack{z'\in\brC_*^{\pi'}(Y'),\,w'\in\crC_*^{\pi'}(Y'), \\
	  \widecheck d(w')=i(z'), \,{\rm Deg}(z')=-k}}(\mdeg(z')-&\mdeg(w'))
	  \leq\\ 
	  &\max( \inf_{\substack{z\in\brC_*^{\pi}(Y),\,w\in\crC_*^{\pi}(Y), \\\widecheck d(w)=i(z), 
	  \,{\rm Deg}(z)=-k}}(\mdeg(z)-\mdeg(w))-\inf_{A}\mathcal E(A),0)+\delta
	\end{align*}
	where the second infimum on the right hand side is over all elements of the moduli spaces 
	$\mathcal M_z^{\overline \pi}(\alpha,\alpha')$ with $\alpha$, $\alpha'$ being irreducible critical points and $z$ being a path 
	along $W$ of index $0$ from $\alpha$ to $\alpha'$.
	By taking the limit of the above inequalities as the norms of perturbations $\pi$, $\pi'$ and $\overline \pi$ converge to zero
	and using Lemma \ref{limit-small-energy}, we can conclude:
	\[
	  \Gamma_{Y'}(k) \leq \max(\Gamma_{Y}(k)-\eta(W),0).
	\]
	In the case that $k$ is a positive integer, the above argument can be modified to obtain the desired stronger inequality 
	using the fact that $L(z)=0$.
\end{proof}

\begin{example}\label{p-q-pqk+1}
	{ There are negative definite cobordisms $W:S^3\to\Sigma(p,q,pqk+1)$ 
	and $W':\Sigma(p,q,pqk+1)\to S^3$ with 
	$b_1=0$ \cite{CG:app-Thm-A}.} Therefore, Theorem \ref{inc-energy} implies that 
	$\Gamma_{\Sigma(p,q,pqk+1)}=\Gamma_{S^3}$.
\end{example}

\begin{cor}\label{hom-cob-simp}
	Suppose $W:Y\to Y'$ is a homology cobordism and $\Gamma_{Y}$ (and hence $\Gamma_{Y'}$) 
	has a finite positive value in its image. Then the inclusions of $Y$ and $Y'$ in $W$ induce non-trivial maps of fundamental groups.
	In particular, if $h(Y)$ is non-trivial, then there is no simply connected homology cobordism from $Y$ 
	to another integral homology sphere.
\end{cor}
\begin{proof}
	Let $W:Y\to Y'$ be a homology cobordism and the fundamental group of $Y$ or $Y'$ map trivially to that of $W$. On one hand, $\Gamma_{Y}=\Gamma_{Y'}$.
	On the other hand, $\eta(W)>0$ and $\Gamma_{Y}$, $\Gamma_{Y'}$ satisfy the inequality given in 
	Theorem \ref{inc-energy}. Thus a positive value in the image of $\Gamma_{Y}$ is a contradiction.
	
	If $h(Y)\neq 0$, then either $h(Y)>0$ or $h(-Y')>0$. Propositions \ref{possible-values-Gamma} and 
	\ref{lower-bound-tau} imply that either $\Gamma_{Y}$ or 
	$\Gamma_{-Y'}$ has a finite positive values in its image. Since
	$W$ can be also regarded as a cobordism from $-Y'$ to $-Y$, it cannot be a simply connected.
\end{proof}

\begin{remark} \label{man-per-ends}
	In his groundbreaking work \cite{Tau:gauge-per}, Taubes proves that if $Y$ is an integral homology sphere and $W:Y\to Y$ 
	is a simply connected homology cobordism (or more generally a definite cobordism), then $Y$ cannot bound a 
	simply connected negative definite smooth manifold with a non-standard intersection form. 
	As it is stated in Theorem \ref{NSIF-neg-def} and will be proved in Section \ref{diag-gen}, if $Y$ bounds a manifold $X$ 
	with non-standard negative definite form (without any assumption
	on $\pi_1(X)$), then $\Gamma_Y(1)$ is a finite positive number. Therefore, Theorem \eqref{inc-energy} gives a new proof
	of Taubes' result.
	
	Taubes' method (gauge theory on manifolds with periodic ends) 
	can be adapted to prove the second half of Corollary \ref{hom-cob-simp}. This was firstly pointed to the author by Chris Scaduto. The author 
	learnt later from Masaki Taniguchi that this method is also used in \cite{Mas:per} 
	and a proof of the second half of Corollary \ref{hom-cob-simp} is implicit there. 
	Given a simply connected homology cobordism $W:Y\to Y$,
	we can from a 4-manifold $M$ with a periodic end and a cylindrical end in the following way. 
	For each non-negative integer $i$, let $W_i$ be a copy of $W$. We fix a metric on $W$ 
	which is cylindrical in a neighborhood of the boundary components
	corresponding to a fixed metric on $Y$. For each positive integer $i$ we identify the outgoing end of $W_{i+1}$ 
	with the incoming end of $W_i$. We also glue a copy of $[0,\infty)\times Y$ to the outgoing end of $W_0$.
	Applying the method of \cite{Tau:gauge-per} to the moduli spaces of ASD connections on $M$, which 
	are asymptotic to the trivial connection 
	on the cylindrical end and have finite energy, shows that $h(Y)$ has to vanish.
\end{remark}

Corollary \ref{hom-cob-simp} implies that the answer to Question \ref{Akbulut-question} is negative if the Rokhlin homomorphism $\mu:\Theta_\Z^3\to \Z/2\Z$ is replaced with $\Gamma_Y$. In general, there are integral homology spheres with non-trivial Rokhlin invariant whose $\Gamma_Y$ does not take any positive value. For example, $\Gamma_{\Sigma(2,3,7)}$ does not take any finite positive value by Example \ref{p-q-pqk+1}. However, $\mu(\Sigma(2,3,7))$ is non-trivial. In \cite{D:spin-Gamma}, we shall show that Corollary \ref{hom-cob-simp} can be extended to other families of integral homology spheres including $\Sigma(2,3,7)$.

\section{Relation to Fintushel and Stern's $R$-invariant} \label{R-inv}

For $n\geq 3$, suppose $a_1$, $\dots$, $a_n$ are relatively prime positive integers. Throughout this section, we denote the Seifert fibered homology sphere $\Sigma(a_1,a_2,\dots a_n)$ by $Y$ unless otherwise is specified. The following elementary results about topology of Seifert fibered homology spheres is standard. (See, for example, \cite{FS:pseudofree}.) The 3-manifold $Y$ admits a standard $S^1$-action and the quotient space is $S^2$. Suppose $W=Y\times D^2/S^1$ where the action of $S^1$ is induced by the Seifert action on $Y$ and the standard action on the 2-dimensional disc $D^2$. Then $W$ has $n$ singular points, one for each special fiber of $Y$. A neighborhood of the $i^{\rm th}$ singular point is given by a cone over the lens space $L(a_i,b_i)$ where the constants $b_i$ satisfy the following identity:
\[
  \sum_{i=1}^{n}\frac{b_i}{a_i}=\frac{1}{a}
\]
and $a=a_1\cdot a_2 \dots a_n$. Let $W_0$ denote the complement of regular neighborhoods of the singular points of $W$. Then $W_0$ is a 4-manifold that:
\[
  \partial W_0=Y\sqcup -L(a_1,\beta_1) \sqcup \dots \sqcup -L(a_n,\beta_n).
\]

There is an obvious projection map from $L=Y\times D^2$ to $W$, that induces a $\U(1)$-bundle on $W_0$. We will write $L_0$ for this $\U(1)$-bundle on $W_0$. Then $w_0=c_1(L_0)$ generates $H^2(W_0;\Z)$ and its restriction to the lens space boundary component $L(a_i,b_i)$ is $b_i$ times the standard generator of the second cohomology of this lens space. Moreover, $W_0$ is negative definite and $c_1(L_0)^2=-\frac{1}{a}$. In particular, if we add cylindrical ends to $W_0$ and form $W_0^+$, then there is an abelian ASD connection $B$ on the bundle $L_0$. We assume that the metric on $W_0^+$ is compatible with the standard spherical metrics on the lens space ends. 
The main goal of this section is to prove Theorem \ref{seifert-space-comp-thm}. At the outset, we mention how $R(a_1,\dots,a_n)$ enters into the proof of this theorem. Suppose $\theta$ is the trivial connection on the the trivial line bundle $\underline \C$ over $W_0$. Then we can form the path\footnote{Previously, we defined paths along a cobordism $W$, which are given by equivalence classes of $\SU(2)$-connections on $W$ with respect to an appropriate equivalence relation. More generally, we can define paths for a pair $(W,w)$ where $W$ is a 4-dimensional cobordism and $w\in H^2(W,\Z)$. The cohomology class $w$ determines a $\U(1)$-bundle on $W$ and we fix a connection $B$ on this bundle. Similar to the case of $\SU(2)$-connections, we define an equivalence relation on $\U(2)$-connections whose central parts are equal to $B$. Any equivalence class of this relation is called a path along $(W,w)$.} $z_0$ along $(W_0,w_0)$ which is represented by the ASD connection $B\oplus \theta$. Then the index formula \cite{MMR:L2-mod-space,Taubes:L2-mod-space} imply that the expected dimension of the moduli space $\mathcal M_{z_0}(W_0,w_0)$, which contains the reducible connection $B\oplus \theta$, is equal to $R(a_1,\dots,a_n)$ of \eqref{Ra1an}. Suppose $\beta_i$ is the unique positive integer less than $a_i$ such that $1+\beta_i\frac{a}{a_i}$ is divisible by $a_i$. Then the formula for $R(a_1,\dots,a_n)$ can be simplified to $2b-3$ where $b$ is given by \cite{NZ:R-com}:
\begin{equation}\label{dim-form-2}
	b=\frac{1}{a}+\sum_{i=1}^n\frac{\beta_i}{a_i}.
\end{equation}

We fix an $\epsilon$-admissible perturbation $\pi$ of the Chern-Simons functional of $Y$. As in Subsection \ref{tilde-func}, we can fix a secondary perturbation $\overline \pi$ of the ASD equation on $W_0^+$ which is compatible with the chosen perturbation of the Chern-Simons functional of $Y$ and the trivial perturbations of Chern-Simons functional of the lens space boundary components of $W_0$. Moreover, we can assume that the norm of the perturbation term $\overline \pi$ is less than $\epsilon$ and all moduli spaces with dimension less than $8$ consist of regular elements \cite{Don:YM-Floer}. The following proposition can be used to prove half of Theorem \ref{seifert-space-comp-thm}.

\begin{lemma}\label{upper-bound-seifert}
	Suppose $b$ is an integer greater than $1$ and $n_0$ denotes $\lfloor \frac{b}{2}\rfloor-1$. Define $\gamma_0\in C_*^\pi(Y) $ as 
	follows:
	\begin{equation}\label{z-element}
	\gamma_0:=\left\{
	\begin{array}{ll}
  	  \displaystyle {\sum_{\alpha} \#\mathcal M_z^{\overline \pi} (W_0,w_0;\alpha) \lambda^{\mathcal E(z)}\alpha}&
	  \hspace{1cm} \text{b is even,}\\
	  \displaystyle {\sum_{\alpha} \#(\mathcal M_z^{\overline \pi} (W_0,w_0;\alpha)\cap V(\Sigma_0)) \lambda^{\mathcal E(z)}\alpha}&
	  \hspace{1cm} \text{b is odd.}
	\end{array}
	\right.  
	\end{equation}	
	where $z$ is a path along $(W_0,w_0)$ that restricts to a generator $\alpha$ of $C_*^\pi(Y)$
	on $Y$ and to the same flat connections on the lens space boundary components as $B\oplus \theta$.
	Moreover, $\Sigma_0$ is an embedded surface representing a generator of $H_2(W)$ and $V(\Sigma_0)$ denotes a divisor representing the homology class $\mu(\Sigma_0)$ on the configuration spaces 
	of connections on $W_0$ \cite{DK}.
	Then $\gamma_0$ satisfies the following properties: 
	\begin{itemize}
		\item[(i)] $d(\gamma_0)=0$;
		\item[(ii)] $D_1\circ U^{k}(\gamma_0)=0$ where $k<n_0$;
		\item[(iii)] $D_1\circ U^{n_0}(\gamma_0)\neq 0$. Moreover, if $\delta$ is an arbitrary positive number, then by choosing
		$\epsilon$ small enough, we have: 
			\begin{equation} \label{mdeg-ineq}
				\mdeg(D_1\circ U^{n_0}(\gamma_0))-\mdeg(\gamma_0)\leq \frac{1}{4a}+\delta.
			\end{equation}	
	\end{itemize}
\end{lemma}

\begin{proof}
	We firstly assume that $b$ is even.
	The coefficient of $\beta$ in $d(\gamma_0)$ can be identified with the number 
	of the boundary points of the compactification of the moduli space:
	\[\mathcal M^{\overline \pi}_{z}(W_0,w_0,\beta)\]
	where $z$ is a path along $(W_0,w_0)$ such that the above moduli space is 1-dimensional, 
	the restriction of $z$ to $Y$
	is $\beta$, and $z$ has the same restriction to the lens space boundary components as $B\oplus \theta$. 
	Here we use the fact that we do not have any boundary component induced by 
	sliding energy off the lens space ends because any flat connection on a lens space is reducible.
	
	For a non-negative integer $k$, if $D_1\circ U^{k}(z_0)$ is non-zero, then there is a $(k+2)$-tuple of ASD connections $(A_0,A_1,\dots,A_k,A_{k+1})$ such that $A_0$ is an ASD connection 
	contributing to the sum in \eqref{z-element}, $A_i$, for $1\leq i\leq k$, is an ASD connection on $\R\times Y$ 
	between irreducible flat connections $\alpha_{i-1}$ and $\alpha_i$ contributing to $U(\alpha_{i-1})$,
	and $A_{k+1}$ is an ASD connection on $\R\times Y$ from $\alpha_{k}$ to $\Theta$ contributing to $D_1(\alpha_{k})$. 
	In particular, the index of $A_0$ is $0$, the index of $A_i$, for $1\leq i \leq k$, is 
	$4$ and the index of $A_{k+1}$ is $1$. Therefore, the sum of the indices of these ASD connections is $4k+1$. We can glue the connections $A_i$ to obtain a connection which has the same asymptotic flat connections
	as $B\oplus \theta$ on the boundary of $W_0$. Therefore, the additivity of ASD index with respect to 
	composition of cobordisms implies that:
	\begin{align*}
		4k+1&=8(\sum_{i=0}^{k+1}\mathcal E(A_i)-\mathcal E(B\oplus \theta))+2b-3\\
		&=8\sum_{i=0}^{k+1}\mathcal E(A_i)+\frac{2}{a}+2b-3
	\end{align*}
	In particular, if $k<n_0$, then the sum of the topological energies of the connections $A_i$ is at most 
	$-\frac{1}{2}+\frac{1}{4a}$.\footnote{In fact, this inequality can be improved to $-1+\frac{1}{4a}$ 
	because $k$ and $n_0$ have the same parity.} Since 
	the connections $A_i$ are solutions of perturbed ASD equation, it is a contradiction if $\epsilon$ is small enough. This verifies the claim in part (ii). In the case that $k=n_0$, the above identity implies that the sum of $\mathcal E(A_i)$ is equal to 
	$\frac{1}{4a}$. Therefore, the inequality in \eqref{mdeg-ineq} holds for a small enough $\epsilon$, if we show that $D_1\circ U^{n_0}(\gamma_0)$ 
	does not vanish.

	Fix a positive integer $k$ smaller than $\frac{b}{2}$. In order to address the last part of the proposition, we introduce 
	$\gamma_k\in C_*^\pi(Y)$ which is defined similar to $\gamma_0$. Fix $n_0$ distinct 
	points $x_1$, $\dots$, $x_{n_0}$ on $W_0$. For any $1\leq i \leq n_0$, a standard construction in Donaldson theory
	allows us to form a co-dimension 4 divisor $V(x_i)$ in the configuration space 
	of irreducible connections on $W_0$, which is a geometric representative for 
	$\mu({\rm point})$ \cite{KM:Polynomial}. Define:
	\begin{equation}\label{z-k}
  	  \gamma_k:=\sum_{\alpha} \#\(\mathcal M_z^{\overline \pi} (W_0,w_0;\alpha)
	  \cap V(x_1)\cap \dots\cap V(x_k)\) \lambda^{\mathcal E(z)}\alpha
	\end{equation}	
	Here we may assume that the divisors are chosen such that the intersection \eqref{z-k} is transversal. 
	In particular, a path $z$ contributes to the above sum, if the dimension of the moduli space
	$\mathcal M_z^{\overline \pi} (W_0,w_0;\alpha)$ is equal to $2k$. 
	A similar argument as in the case of $\gamma_0$ shows that $\gamma_k$ 
	is a cycle. 
	Moreover, we can show that $\gamma_{k+1}$ and $U(\gamma_k)$ differ by a co-boundary element
	in $C_*^\pi(Y)$. To see this, we allow $x_{k+1}$ to move off the boundary component $Y$ of $W_0$ and 
	consider the associated 1-parameter family of moduli spaces. Cutting this moduli space by the divisors
	$V(x_1)$, $\dots$, $V(x_k)$ and studying its ends give the desired relation between 
	$\gamma_{k+1}$ and $U(\gamma_k)$.\footnote{The essential points here are the fact that 
	any flat connection on a lens space is reducible
	and we also do not face any reducible connection on $W_0$ in our analysis of the relevant moduli spaces.} 
	In particular, $D_1U^{n_0-k}(\gamma_k)$ is independent of $k$.
	
	Consider the moduli space $\mathcal M^{\overline \pi}_{\gamma_0}(W_0,w_0)$ 
	which contains the class of the reducible connection $B\oplus \theta$. 
	The divisors $V (x_i)$ is defined only on the complement of the 
	reducible connection. In fact,
	a neighborhood of the reducible connection is a cone over the projective space  ${\bf CP}^{b-2}$ 
	and the restriction of $V(x_i)$ to the boundary of this cone represents the cohomology class $-h^2$
	where $h$ is the generator of $H^2({\bf CP}^{b-2})$ \cite[Subsection 5.1.2]{DK}. 
	Let $\mathcal M'$ be the complement of this neighborhood of the reducible element of 
	$\mathcal M_{\gamma_0}^{\overline \pi}(W_0,w_0)$ and $N'$ denotes the boundary of this neighborhood. 
	Therefore, we can form
	the 1-dimensional moduli space:
        \[
          Z=\mathcal M'\cap V(x_1) \cap \dots \cap V(x_{n_0}).
        \]	
	We compactify $Z$ in the standard way to form a compact 1-manifold whose boundary is given by 
	$N'\cap V(x_1) \cap \dots \cap V(x_j)$ and
        \begin{equation} \label{bdry-type-2}
          \coprod_{\gamma\#\gamma'=\gamma_0} (\mathcal M_{\gamma} (W_0,w_0;\alpha)
          \cap V(x_1) \cap \dots \cap V(x_j))\times\breve {\mathcal M}_{\gamma'}(\alpha,\Theta) 
        \end{equation}
	Since the count of points in $N'\cap V(x_1) \cap \dots \cap V(x_j)$ does not vanish, the count of elements in 
	\eqref{bdry-type-2} is also non-zero. The latter count is equal to $D_1(\gamma_{n_0})$ by definition.
	Thus we conclude that $D_1\circ U^{n_0}(\gamma_0)$ is nonzero. 
	Analogous arguments can be used to prove similar claims in the case that $b$ is odd. 
	The only required new ingredient is that the restriction of the cohomology class $V(\Sigma_0)$ to 
	${\bf CP}^{b-2}$ is a non-zero multiple of $h$. (In fact, it is equal to 
	$-\langle c_1(L_0),\Sigma_0\rangle h$ \cite[Subsection 5.1.2]{DK}.)
\end{proof}

\begin{prop}\label{seifert-space-comp-prop}
	For $Y$ as above and $1\leq i \leq \lfloor \frac{b}{2}\rfloor$, we have:
	\[
	  \Gamma_{Y}(i)=\frac{1}{4a}.
	\]
\end{prop}

\begin{proof}
	Lemma \ref{upper-bound-seifert} asserts that $\Gamma_Y(i)\leq \frac{1}{4a}$. It is shown in \cite{FS:HFSF} that the value of the 
	Chern-Simons functional for $\SU(2)$-flat connections on $Y$ has the form $\frac{l}{4a}$ where $l$ is an 
	integer\footnote{In fact, the 4-manifold $W_0$ plays a key role in this part, too. 
	The main point is that any $\SU(2)$ flat connection on $Y$
	extends to a flat connection on $W_0$.}. (See also \cite{Au:CS-comp}.) 
	Since $\Gamma_Y(i)$ is positive by Proposition \ref{lower-bound-tau}, Proposition \ref{CS-val} implies that 
	$\Gamma_Y(i)\geq \frac{1}{4a}$. This completes the proof.
\end{proof}

Theorem \ref{seifert-space-comp-thm} is a generalization of  Proposition \ref{seifert-space-comp-prop} to the case that we consider a connected sum of Seifert fibered spaces with positive $R$-invariants:

\begin{proof}[Proof of Theorem \ref{seifert-space-comp-thm}]
	Let $M$ be the standard cobordism from the disjoint union of 3-manifolds $Y_i$ to $Y_1\#\dots \# Y_k$ 
	obtained by gluing 1-handles.
	For each $i$, we follow the above construction to construct a cobordism $W_{i,0}$ from a 
	disjoint union of lens spaces to $Y_i$.
	We can glue $W_{1,0}$, $\dots$, $W_{k,0}$ to $M$ and obtain $\widetilde M$ a cobordism from a 
	disjoint union of lens spaces to $Y_1\#\dots \# Y_k$. For each $i$, we also constructed a $\U(1)$-bundle $L_{i,0}$ 
	on $W_{i,0}$. We extend the bundle $L_{1,0}$
	on $W_{1,0}$ trivially to $\widetilde M$ and denote the resulting bundle by $\widetilde L$. 
	By applying the same argument as in the proof of Lemma \ref{upper-bound-seifert} where $W_0$ and $L_0$ 
	are replaced with $\widetilde M$ and $\widetilde L$, we can show that for 
	$1\leq i \leq \lfloor \frac{R(a_1,\dots,a_n)+3}{4}\rfloor$:
	\[
	  \Gamma_{Y}(i)=\frac{1}{4a_1a_2\dots a_n}.
	\]
	The reverse inequality is also a consequence of Lemma \ref{lem-ineq-D1-tau} below.
\end{proof}

\begin{lemma}\label{lem-ineq-D1-tau}
	For integral homology spheres $Y$ and $Y'$, we have:
	\begin{equation} \label{ineq-D1-tau}
	  \Gamma_{Y\#Y'}(1)\geq \min\{\tau(Y),\tau(Y')\}
	\end{equation}
\end{lemma}
This lemma strengthens Lemma \ref{lower-bound-tau} for integral homology spheres which are homology cobordant to connected sum of two integral homology spheres.

\begin{proof}
	Consider the standard cobordism $N$ from $Y\#Y'$ to $Y\sqcup Y'$ given by gluing a 3-handle. 
	Let $\epsilon_i$ be a sequence of positive real numbers converging to $0$.
	We fix $\epsilon_i$-admissible perturbations $\pi$, $\pi'$ and $\pi^\#$ on $Y$, $Y'$ and $Y\#Y'$ 
	and extend them to an $\epsilon_i$-admissible perturbation $\overline \pi$ on $N$. Suppose 
	$\gamma=\sum_{i=1}^{k}s_i\lambda^{r_i}\alpha_i$ is an element of $\mathcal L_1^{\pi^\#}(Y\#Y')$. 
	For each $i$, there is a path $z_i$ of index $1$ along $N$ from the connection $\alpha_i$ to the trivial
	connections $\Theta$ and $\Theta'$ on $Y\sqcup Y'$. 
	Let $\mathcal M^{\overline \pi,+}_{z_i}(N,\alpha_i,\Theta,\Theta')$ be the standard compactification of the 
	1-dimensional moduli space 
	$\mathcal M^{\overline \pi}_{z_i}(N,\alpha_i,\Theta,\Theta')$. Clearly, we have:
	\[
	  \sum_{i=1}^ks_i\lambda^{r_i+\mathcal E(z_i)}\cdot
	  \#\partial (\mathcal M^{\overline \pi,+}_{z_i}(N,\alpha_i,\Theta,\Theta'))=0.
	\]
	We can also analyze the boundary components of the moduli space 
	$\mathcal M^{\overline \pi,+}_{z_i}(N,\alpha_i,\Theta,\Theta')$ 
	to show that the above identity implies that:
	\[
	  F\circ d(\gamma)+D_1(\gamma)+\sum_i s_i\lambda^{r_i+\mathcal E(z_i)} \sum_{x_1\#x_2=z_i}
	  \mathcal M_{x_1}^{\overline \pi}(N,\alpha_i,\beta,\Theta)\times 
	  \breve{\mathcal M}_{x_2}^{\pi}(\R\times Y,\beta,\Theta)
	  +\hspace{2cm}
	 \]
	 \begin{equation} \label{co-mult-idenity}
	  \hspace{3cm}\sum_i s_i\lambda^{r_i+\mathcal E(z_i)} \sum_{y_1\#y_2=z_i}\mathcal 
	  M_{y_1}^{\overline \pi}(N,\alpha_i,\Theta,\beta')
	  \times \breve{\mathcal M}_{y_2}^{\pi'}(\R\times Y',\beta',\Theta)=0
	\end{equation}
	Here $F:C_*^{\pi^\#}(Y\#Y')\to \Lambda$ is induced by counting solutions of 
	the perturbed ASD equation on $N$ corresponding to paths of index $0$ along $N$ which restrict to an irreducible
	connection on $Y\#Y'$ and trivial connections on $Y$ and $Y'$. 
	The path $x_1$ (respectively, $y_1$) along $N$ has index $0$ and restricts to 
	the trivial connection on the end $Y'$ (respectively, $Y$) and to irreducible connections
	on the remaining ends. The path $x_2$ along $\R \times Y$ (respectively, $y_2$ along $\R \times Y'$) has index $1$
	and restricts to an irreducible connection on the incoming end and to the trivial connection
	on the outgoing end. Since $d(\gamma)=0$ and $D_1(\gamma)\neq 0$, \eqref{co-mult-idenity} implies that 
	either there is a path $x_2$ along $\R\times Y$ such that the moduli space 
	$\mathcal M_{x_2}^{\pi}(\R\times Y,\beta,\Theta)$ contains an element $A_i$ with:
	\[
	  \mdeg(D_1(\gamma))-\mdeg(\gamma)\geq \mathcal E(A_i)-\delta
	\]
	or there is a path $y_2$ along $\R\times Y'$ such that the moduli space 
	$\mathcal M_{y_2}^{\pi'}(\R\times Y',\beta',\Theta')$ contains an element $A_i'$ with:
	\[
	  \mdeg(D_1(\gamma))-\mdeg(\gamma)\geq \mathcal E(A_i')-\delta.
	\]
	Here $\delta$ is a constant which converges to $0$ as $i\to \infty$. 
	Now by letting $i$ go to $\infty$ and using Lemma \ref{tau-pos}, we can verify \eqref{ineq-D1-tau}.
\end{proof}

\begin{remark}
	Theorem \ref{seifert-space-comp-thm} and Lemma \ref{lem-ineq-D1-tau} give some instances 
	of relations among $\Gamma_{Y\#Y'}$, $\Gamma_Y$ and $\Gamma_{Y'}$. 
	In  \cite{D:conn-Gamma}, we study this relationship more systematically. 	
\end{remark}

\begin{cor}[\cite{Fur:hom-cob}]\label{linear-ind}
	Suppose $\{Y_i=\Sigma(a_{i,1}, \dots, a_{i,n_i})\}_{i\in I}$ is a collection of Seifert fibered homology spheres with 
	positive $R$-invariants such that the positive integers $a_i:=a_{i,1} \cdot a_{i,2} \dots a_{i,n_i}$ are distinct.
	Then the integral homology spheres $Y_i$ determine linearly independent elements of $\Theta_\Z^3$.	
\end{cor}

\begin{proof}
	If there is a linear relation among $Y_i$ we can assume that we have a relation of the following form:
	\[
	  n_1 Y_{i_1}+\dots+n_{k}Y_{i_k}=m_1 Y_{j_1}+\dots+m_{l}Y_{j_l}
	\]
	such that all integers $n_r$ and $m_s$ are positive, and $i_r$ and $j_s$ are all distinct elements of $I$. Evaluating
	the invariant $\Gamma$ at $1$ for the 3-manifolds given above implies that:
	\[
	  \frac{1}{4\max\{a_{i_1}, \dots,a_{i_k}\}}= \frac{1}{4\max\{a_{j_1}, \dots,a_{j_l}\}}
	\]
	which is in contradiction with the assumption that the integers $a_i$ are distinct.
\end{proof}

\begin{example}\label{pqk-1}
	If $p$, $q$ are coprime positive integers and $k$ is another positive integer, then $R(\Sigma(p,q,pqk-1))=1$. Therefore, Proposition \ref{seifert-space-comp-prop} implies that:
	\[\Gamma_{\Sigma(p,q,pqk-1)}(1)=\frac{1}{4pq(pqk-1)}.\]
	Corollary \ref{linear-ind} implies that $\{\Sigma(p,q,pqk-1)\}_{k\in \Z^{> 0}}$ spans a $\Z^{\infty}$ subgroup of 
	$\Theta^3_\Z$. This result is proved by Furuta \cite{FS:HFSF} and 
	Fintushel and Stern \cite{Fur:hom-cob}.
\end{example}

\begin{example}\label{pqk+1}
	If $p$, $q$,  $k$ are as in Example \ref{pqk-1}, then $R(p,q,pqk+1)=-1$. Therefore, Proposition 
	\ref{seifert-space-comp-prop} does not say anything about $\Gamma_Y$ when $Y=\Sigma(p,q,pqk+1)$. 
	This is in line with Example \ref{p-q-pqk+1} where it is shown that $\Gamma_Y=\Gamma_{S^3}$.
\end{example}

\begin{example}\label{large-R}
	Suppose $p$ and $q$ are coprime numbers, and $a_1$, $a_2$, $\dots$, $a_{2n+1}$ is a sequence of positive 
	integers which $a_1=p$, $a_2=q$ and for $i\geq 3$:
	\[
	  a_i=k_ia_1a_2\dots a_{i-1}-1.
	\]
	Here $k_i$ is an arbitrary positive integer. Then $R(a_1,\dots,a_{2n+1})=n$. In particular, $R$-invariant 
	could be arbitrarily large. As it is discussed in the introduction, if $R(a_1,\dots,a_{2n+1})\geq 5$, then 
	Propositions \ref{tau'-str} and \ref{seifert-space-comp-prop} imply that
	$\Sigma(a_1,\dots,a_{2n+1})$ is not homology cobordant to an $\SU(2)$-non-degenerate integral homology sphere.
\end{example}

Corollary \ref{whitehead-dble} provides another family of 3-manifolds that similar techniques can be used to obtain information about $\Gamma_Y$. Next, we give the proof of this corollary which recasts the arguments of \cite{KH:Wh-dble}.
\begin{proof}[Proof of Corollary \ref{whitehead-dble}]
	There is a negative definite cobordism with trivial first integral homology 
	from $\Sigma(p,q,2pq-1)$ to $Y_{p,q}$ \cite[Lemma 3.6]{KH:Wh-dble}. Thus Theorem \ref{neg-def} and Example \ref{pqk-1} imply that: 
	\[\Gamma_{Y_{p,q}}(1)\leq \Gamma_{\Sigma(p,q,2pq-1)}(1)=\frac{1}{4pq(2pq-1)}.\]
	The study of the Chern-Simons functional of $Y_{p,q}$ in \cite{KH:Wh-dble} shows that $\tau(Y_{p,q})\geq \frac{1}{4pq(4pq-1)}$, which can be used to verify the other inequality in \eqref{Ypq-ineq}.
	The proof of the more general inequalities in \eqref{Ypq-ineq-gen} is similar. (See the proof of 
	Theorem \ref{seifert-space-comp-thm}.) 
\end{proof}

\section{4-manifolds and Reducible Connections}\label{diag-gen}

The following theorem is the counterpart of \cite[Theorem 3 and Proposition 1]{Fro:h-inv} which generalizes Donaldson's celebrated diagonalizability theorem \cite{Don:neg-def-gauge}:
\begin{prop} \label{bounding-reducible}
	Suppose $Y$ is a homology sphere, and $X$ is a 4-manifold with boundary $Y$ such that $b_1(X)=0$, and 
	the intersection form $Q$ of $X$ on $H^2(X,\Z)/{\rm Tor}$ is negative-definite. 
	Suppose a cohomology class $e \in H^2(X,\Z)$ is fixed such that $Q(e)$ is an even integer, $|Q(e)| \geq 2$, 
	$|Q(e)|\leq |Q(e')|$ for any $e'$ with $e'\equiv e\mod 2$,
	and:
	\begin{equation}\label{red-ends}
		\sum (-1)^{(\frac{e+e'}{2})^2}\neq 0.
	\end{equation}
	where the sum is over all pairs $\{e',-e'\}$ such that $e'\in H^2(X,\Z)$, $e'\equiv e\mod 2$, and $Q(e)=Q(e')$.
	Then $\Gamma_Y(n_0)\leq -\frac{1}{4}Q(e)$ where $n_0=-\frac{1}{2}Q(e)$. 
	The equality holds only if the fundamental 
	group of $X$ admits a representation to $\SU(2)$ with non-trivial restriction to the boundary.
\end{prop}

\begin{proof}
	We follow the proof of \cite[Proposition 1]{Fro:h-inv}. Since the proof is also similar to the proof of 
	Lemma \ref{upper-bound-seifert}, we only sketch the main steps. 
	Suppose $L$ is the complex line bundle on $X$ which 
	represents the cohomology class $e$. We also fix a metric with cylindrical ends on $X$.
	The bundle $L$ admits an ASD connection $B$ such that $|\!|F(B)|\!|_{L^2(X)}$ is finite. 
	Let $\epsilon_i$ be a sequence of positive numbers converging $0$, and for each $i$, fix an $\epsilon_i$-admissible
	perturbation $\pi_i$ of the Chern-Simons functional of $Y$. We also pick a compatible
	perturbation $\overline \pi_i$ of the ASD equation on $X$ { as in the previous section.}
	Define:
	\[
	  \gamma:=\sum_{z}\#\mathcal M^{\overline \pi_i}_z(X,e;\alpha) 
	  \lambda^{\mathcal E(z)}\alpha
	\]
	where the sum is over all paths along $(X,e)$ of index $0$ which restricts to a generator $\alpha$ of $C_*^{\pi_i}(Y)$.
	
	Following the same argument as in the proof of Lemma \ref{upper-bound-seifert}, we can show that $ d(\gamma)=0$
	and $D_1U^{k}(\gamma)=0$ for $k<n_0-1$. Moreover, the inequality in \eqref{red-ends} implies that 
	$D^1U^{n_0-1}(z_0)$ is a non-zero multiple of $\lambda^{-\frac{1}{4}e^2}$. 
	To be a bit more detailed, we consider the moduli space that contains the ASD connection $B\oplus \theta$ and 
	as in Lemma \ref{upper-bound-seifert}, we cut down the complement 
	of a small neighborhood of the reducible connections with $n_0-1$ codimension $4$ divisors representing $\mu({\rm pt})$. 
	This determines a 1-dimensional moduli space. Clearly, 
	the signed count of the boundary points of this moduli space is zero.
	This count has contribution from $D^1U^{n_0-1}(z_0)$ and reducible connections. 
	The count associated to the reducible 
	connections is a non-zero multiple of \eqref{red-ends}, and hence it is non-zero. 
	Consequently, $D^1U^{n_0-1}(z_0)$ is non-zero.
	The dimension formula implies that this non-zero number is a multiple of $\lambda^{-\frac{1}{4}Q(e)}$. 
	There is also an element $A_i$ of the moduli space of the form $\mathcal M^{\overline \pi_i}_z(X,e;\alpha)$
	such that $\mdeg(\gamma)$ is equal to $\mathcal E(A_i)$. In particular, we have:
	\[
	  \mdeg(D_1U^{n_0-1}(\gamma))-\mdeg(\gamma)=-\frac{1}{4}Q(e)-\mathcal E(A_i)
	\]
	By letting $i\to \infty$, we have:
	\[
	  \Gamma_Y(n_0)\leq -\frac{1}{4}Q(e)-\limsup_i\mathcal E(A_i)
	\]
	Analogue of Lemma \ref{limit-small-energy} for 4-manifolds with one boundary component implies that 
	there is a (possibly broken) ASD connection $A_0$ on $X$ which is asymptotic to an irreducible flat connection 
	on $Y$ such that:
	\[
	  \Gamma_Y(n_0)\leq -\frac{1}{4}Q(e)-\mathcal E(A_0)
	\]	
	In particular, $\Gamma_Y(n_0)\leq -\frac{1}{4}Q(e)$ and if the equality holds then $A_0$ has to be flat. Therefore, 
	there is an $\SU(2)$-representation of $\pi_1(X)$, which extends a non-trivial representation of $\pi_1(Y)$.	
\end{proof}

\begin{prop} \label{bounding-reducible-2}
	Suppose $Y$ and $X$ are given as in Proposition \ref{bounding-reducible}. Suppose also 
	a non-negative integer $m$, a homology class $\Xi \in H_2(X,\Z)$ and a cohomology class 
	$e \in H^2(X,\Z)$ are fixed such that $Q(e) \equiv m$ mod 2, $|Q(e)| \geq 2$, $|Q(e)|\leq |Q(e')|$ for any $e'$ with 
	$e'\equiv e$ mod $2$ and:
	\begin{equation}
		\sum (-1)^{Q(\frac{e+e'}{2})}(\Xi \cdot e')^m\neq 0.
	\end{equation}
	where the sum is over all pairs $\{e',-e'\}$ such that $e'\in H^2(X,\Z)$, $e'\equiv e\mod 2$, and $Q(e)=Q(e')$.
	Then $\Gamma_Y(n_0)\leq -\frac{1}{4}Q(e)$ where $n_0=-\frac{Q(e)+m}{2}$. 
	The equality holds only if the fundamental 
	group of $X$ admits a representation to $\SU(2)$ with non-trivial restriction to the boundary.
\end{prop}
\begin{proof}
	Essentially the same proof as in Proposition \ref{bounding-reducible} verifies this claim. 
	We just need to modify the definition 
	of $\gamma$ as follows:
	\[
	  \gamma:=\sum_{z}\#(\mathcal M_\gamma^{\overline \pi_i}(X,e;\alpha)\cap V(\Sigma_1)\cap \dots  V(\Sigma_m)) 
	  \lambda^{\mathcal E(\gamma)}\alpha
	\]
	where $\Sigma_i$ is an embedded surface\footnote{As it is customary in the definition of Donaldson invariants,
	we assume that these surfaces intersect generically to avoid the issue of instanton bubbles.} 
	in $X$ representing the cohomology class $\Xi$ and 
	$V(\Sigma_i)$ is the standard codimension $2$ divisor in the configuration space of irreducible connections on $X$ 
	representing $\mu(\Sigma)$ \cite{DK}.
\end{proof}

\begin{proof}[Proof of Theorem \ref{NSIF-neg-def}]
	{
	Let $e$ be a cohomology class which represents a non-zero element of with minimum $|Q(e)|$. In particular,
	$Q(e)=-m(\mathcal L)$ where $m(\mathcal L)$ is defined in \eqref{m(l)}. If $e'\in H^2(X,\Z)$ is another
	cohomology class that $e\equiv e'$ mod $2$ and $Q(e)=Q(e')$, then either $e-e'$ or $e+e'$ is a torsion element. 
	In particular, in the case that $Q(e)$ is an even number, the condition in \eqref{red-ends} holds and 
	we can conclude that $\Gamma_Y(-\frac{1}{2}Q(e))\leq -\frac{1}{4}Q(e)$. In the case that $Q(e)$ is an odd number,
	we set $m=1$ and $\Xi$ to be a cohomology class that $\Xi\cdot e=1$ and apply 
	Proposition \ref{bounding-reducible-2} to conclude that $\Gamma_Y(-\frac{Q(e)+1}{2})\leq -\frac{1}{4}Q(e)$.
	This inequality completes the proof.}
\end{proof}

\appendix
\section{Evaluating a Function at Cohomology classes} \label{min-max-app}
Suppose $M$ is a compact topological space, $f:M\to \R$ is a continuous function and $\sigma\in H_*(M,\Z)$. Let $M(r):=f^{-1}((-\infty,r])$ and define: 
\[
  f(\sigma):=\inf \{r\mid \sigma \in {\rm image}(i_*:H_*(M(r),\Z)\to H_*(M,\Z))\}
\]
We say $f(\sigma)$ is the {\it evaluation} of $f$ at the homology class $\sigma$. 

In the case that $M$ is a smooth manifold and $f$ is a smooth function, we can reformulate the above definition using the language of Morse homology. We firstly assume that $f$ is a Morse function and fix a Riemannian metric on $M$ such that $f$ is a Morse-Smale function. Suppose $(C_*(f),d)$ is the Morse complex associated to $f$. Then $C_*(f)$ is generated by critical points of $f$, and $d$ is defined using down-ward gradient flow lines of $f$. Moreover, the homology of the complex $(C_*(f),d)$, denoted by $H_*(f)$, is naturally isomorphic to $H_*(M,\Z)$. For each homology class $\sigma$, we have:
\begin{equation}\label{min-max}
  f(\sigma)=\min_{\sigma=[\sum_{i}a_i\alpha_i]} \max_i\{f(\alpha_i)\}
\end{equation}
Here the minimum is over all linear combinations $\sum_{i}a_i\alpha_i$ of critical points of $f$ which represents the homology class $\sigma$. The above identity is a consequence of the fact that for any $r$, the sub-complex generated by critical points $\alpha$ of $f$ with $f(\alpha)\leq r$ has the same homology as the singular homology of $M(r)$. Clearly, the above min-max formula is independent of the choice of the metric because it is equal to $f(\sigma)$. Alternatively, one can use standard continuation maps to show the independence of the left hand side of \eqref{min-max} from the choice of the metric. 

\begin{example}
	If $f$ is a self-indexing Morse function, then for any $\sigma\in H_i(M,\Z)$, we have $f(\sigma)=i$.
\end{example}

The above definition can be modified easily for the functions which are not necessarily Morse. Given a function $f$, we fix a smooth function $\pi_i:M\to \R$ for each positive integer $i$ such that $f+\pi_i$ is a Morse function and $|\pi_i|_{C^0}\to 0$. Then we have:
\begin{equation}\label{min-max-per}
  f(\sigma)=\lim_{i\to \infty}\min_{\substack{\sum_{i}a_i\alpha_i\in C_*(f+\pi_i)\\\sigma=[\sum_{i}a_i\alpha_i]}} \max_i\{f(\alpha_i)\}
\end{equation}
In particular, the term on the min-max formula on the right hand side of \eqref{min-max-per} is independent of the choice of the perturbations $\pi_i$. One can also observe from the definition of $f(\sigma)$ that this number is equal to a critical value of $f$.

The definition of $\Gamma_Y$ has some formal similarities with the expression on the left hand side of \eqref{min-max-per}, where $M$ is replaced with the configuration space $\mathcal B(Y)$ and $f$ is replaced with the $S^1$-valued Chern-Simons functional $\CS$. Another similarity is that $\Gamma_Y$, modulo integers, essentially takes values in the set of critical values of the Chern-Simons functional, i.e., the values of the Chern-Simons functional on $\SU(2)$-flat connections. However, one should not go too far in this analogy. To have a better finite dimensional approximation to the construction of this paper one should consider a manifold $M$ with an $\SO(3)$ action which has one fixed point and otherwise it is free. We also need to fix an $\SO(3)$-invariant function on $M$. Working with an appropriate model for the homology of $M$, we can produce a chain complex with coefficients in $\Lambda$ which has the form in \eqref{shape-SO-com}. Then repeating the construction of Subsections \ref{equiv} and \ref{GammaY} would provide a better finite dimensional approximation to $\Gamma_Y$.

\bibliography{references}
\bibliographystyle{hplain}
\Addresses
\end{document}